\documentclass[11pt, reqno]{amsart}    

\usepackage{amsmath,amssymb,latexsym}
\usepackage{color}
\usepackage{xcolor}
\usepackage{graphicx}
\usepackage{cite}
\usepackage[colorlinks=true]{hyperref}
\hypersetup{urlcolor=blue, citecolor=red, linkcolor=blue}

\newcommand{\R}{\mathbb{R}}

\newcommand{\bigO}{\mathcal{O}}
\newcommand{\diff}{\mathrm{d}}

\newtheorem{theorem}{Theorem}[section]
\newtheorem{lemma}[theorem]{Lemma}

\newtheorem{remark}[theorem]{Remark}

\newtheorem*{main-theorem}{Main Theorem}
\newtheorem*{remark*}{Remark}
\newtheorem*{lemma*}{Lemma A.1}

\numberwithin{equation}{section}

\begin{document}

\title[Wave breaking for Whitham-type equations]{The wave breaking for Whitham-type equations revisited}

\author{Jean-Claude Saut}

\author{Yuexun Wang}

\address{ Universit\' e Paris-Saclay, CNRS, Laboratoire de Math\'  ematiques d'Orsay, 91405 Orsay, France.}
\email{jean-claude.saut@universite-paris-saclay.fr}

\address{	
	School of Mathematics and Statistics, 
	Lanzhou University, 370000 Lanzhou, China.}
\address{Universit\' e Paris-Saclay, CNRS, Laboratoire de Math\'  ematiques d'Orsay, 91405 Orsay, France.}

\email{yuexun.wang@universite-paris-saclay.fr}

\thanks{}


\begin{abstract} We prove wave breaking (shock formation) for some Whitham-type
	equations which include the Burgers-Hilbert equation, the fractional Korteweg-de Vries equation, and the classical  Whitham equation. The result seems to be new for the Burgers-Hilbert equation. In the other cases we provide  simpler proofs than the known ones.
\end{abstract}
\maketitle

\section{Introduction}
We consider  nonlocal dispersive perturbations of the Burgers equation 

\begin{equation}\label{Whitype}
u_t+uu_x+\int_{-\infty}^{\infty}k(x-y)u_x(y,t)dy=0,
\end{equation}
where $k$ is a real-valued kernel measuring the (weak) dispersive effects.
This equation can also be written in the form
\begin{equation*}\label{Whibis}
u_t+uu_x-Lu_x=0,
\end{equation*}
where the Fourier multiplier operator $L$ is defined by 

$$\widehat{Lf}(\xi)=p(\xi)\hat{f}(\xi),$$
where $p=\hat{k}.$

A particular case is the fractional KdV equation (fKdV)
\begin{align}\label{eq:main-1}
\partial_t u+u\partial_xu-|D|^\alpha \partial_xu=0,
\end{align}
where \(u\) maps \(\R_t\times\R_x\) to \(\R\) and \(|D|^\alpha\) is the usual Fourier multiplier operator with  symbol \(|\xi|^\alpha\).

We will restrict for the fKdV equation to the {\it weakly dispersive} case $-1<\alpha<0$, and refer to \cite{LPS2, MPV} for a study of the case $0<\alpha<1$ which displays quite different (dispersive) properties.

When \(\alpha=-1\), \eqref{eq:main-1}  is the Burgers-Hilbert equation introduced in \cite{BH} as a model for waves with constant nonzero linearized frequency   providing an  effective equation for the motion of a vorticity discontinuity in a two-dimensional  flow of an inviscid,  incompressible fluid:
\begin{align}\label{eq:BH}
u_t+uu_x-\mathcal H u=0,
\end{align}
where $\mathcal H= \text{p.v.}\; \frac{1}{x}$ is the Hilbert transform with Fourier symbol $-i\text{sgn}\;\xi.$

We will also consider the Whitham equation introduced in \cite{Whi}
\begin{equation}\label{Whit}
\partial_tu+u\partial_xu+\int_\R K(x-y)\partial_y u(y,t)dy=0,
\end{equation}
where 
\begin{align*}
K(x)=\frac{1}{\sqrt{2\pi}}\int_\R e^{\mathrm{i}x\xi}\sqrt{\frac{\tanh \xi}{\xi}}\,\diff \xi,
\end{align*}
and its rescaled version 
\begin{align}\label{Whitresc}
\partial_t u+\epsilon u\partial_x u+\int_\R K_\epsilon(x-y)\partial_y u(y,t)\,\diff y=0, 
\end{align}
where  \(K_\epsilon\) is defined by
\begin{align}\label{res-kernal}
K_\epsilon(x)=\frac{1}{\sqrt{2\pi}}\int_\R e^{\mathrm{i}x\xi}\sqrt{\frac{\tanh \sqrt{\epsilon}\xi}{\sqrt{\epsilon}\xi}}\,\diff \xi=\frac{1}{\sqrt{\epsilon}}K\left(\frac{x}{\sqrt{\epsilon}}\right),
\end{align}
$\epsilon$ being a positive parameter, the long wave limit $\epsilon\to 0$ making the link with the KdV equation, see \cite{KLPS} and Section \ref{sec7} below.

Although of not clear physical relevance (see however \cite{KLPS} for a rigorous connection of the Whitham equation to the modeling of weakly nonlinear water waves), the fKdV equation \eqref{eq:main-1} when $-1<\alpha<0$ and the Whitham equation \eqref{Whit} (likes the fKdV equation with $\alpha=-\frac{1}{2}$ for high frequencies) are very rich toy models to investigate the effects of a weak dispersive term on the dynamics of a conservation law such as the Burgers equation
\begin{equation}\label{Bur}
\partial_t u+\epsilon u\partial_x u=0, \quad u(x,0)=\phi(x).
\end{equation}

It is well known that for the Burgers equation \eqref{Bur}  any non trivial non increasing $H^2$ initial data $\phi$ of size $O(1)$ will lead to a shock formation at a finite time of order  $O\left(\frac{1}{\epsilon ||u_0||_{H^2}}\right).$

The aim of the present paper is to investigate the effect of adding a weak dispersion on this phenomenon, in particular to see if the shock formation persists.  Such a question was already raised for the first time by Whitham in \cite{W} and  this issue has been considered in previous works that we describe now.  

Naumkin and Shishmarev \cite{NS} and Constantin and Escher \cite{CE} have proven a wave breaking phenomena for a Whitham type equation such as \eqref{Whitype} with a kernel $k$ satisfying 
\begin{equation*}
k\in C(\R)\cap L^1(\R), \;\text{symmetric and monotonically decreasing on}\; \R_+.
\end{equation*}

This result does not apply to the Whitham equation \eqref{Whit} since the Whitham kernel satisfies  
$K(0)=\infty.$

When $-1<\alpha<0,$ Castro, C\' ordoba and Gancedo have proven for the fKdV equation \eqref{eq:main-1} with some initial data in $L^2(\R)\cap C^{1+\delta}(\R)$ with   $0<\delta<1,$ a finite time blow-up of the $C^{1+\delta}(\R)$ norm, without proving the occurrence of a wave breaking, that is blow-up of the sup norm of gradient of the solution, the solution itself remaining bounded.

To our knowledge no rigorous proof of shock formation for the Burgers-Hilbert equation \eqref{eq:BH} has been established although the numerical simulations in \cite{BH} strongly suggest its existence.

Finally, the existence of a wave breaking for the fKdV equation \eqref{eq:main-1} when
$-1<\alpha<-\frac{1}{3}$ and for the Whitham equation \eqref{Whit} has been established in \cite{HT, Hur}. 

We refer to \cite{MR3317254, KLPS} for various numerical simulations of the fKV equation and Whitham equations, in particular for a description of the blow-up solutions.

We aim in this paper to provide a simple proof of wave breaking for the fKdV equation in the all range $-1\leq \alpha<-\frac{2}{5}$ (including thus the Burgers-Hilbert equation) and for the Whitham equation.
Contrary to \cite{HT,Hur} our proof does not use an infinite number of ODE's and hence less assumptions on the initial data are needed. The wave breaking for the Burgers-Hilbert equation \eqref{eq:BH} is not covered in \cite{HT,Hur} whose argument causes a logarithmic loss in estimating the term \(K_1(t,x)\) (see \eqref{BH-kernal}) which prevents the proof to work. We overcome this difficulty by using the cancellation property of Hilbert transform and Morrey's inequality to replace the integration by parts.  
For the Whitham equation \eqref{Whit}, our new observation is that one can use interpolation between \(\|\partial_x^3u\|_{L^2}\) (the singularity is too high which can not be used directly) and \(\|\partial_xu\|_{L^\infty}\) (the singularity is low) to control \(\|\partial_x^2u\|_{L^\infty}\), this balance allows us to give a very simple proof. This idea also works for the fKdV equation \eqref{eq:main-1} in the range $-\frac{1}{2}\leq \alpha<-\frac{2}{5}$. We remark that one could push \(\alpha\) forward up to $\alpha =0$  if one could obtain better estimates on higher derivatives of the solution. 

For the rescaled Whitham equation \eqref{Whitresc} we show that its wave breaking time has the order \(\bigO\big(\epsilon^{-1}[-\inf_{\R}\phi^\prime(x)]^{-1}\big)\) which confirms that the long-time existence in \cite{KLPS} is optimal.

\section{The main results}	
To present our main results, we will use the
best constants from Gagliardo-Nirenberg interpolation, Sobolev embedding and Morrey embedding inequalities
\begin{equation*}
\begin{aligned}
&C_{\mathrm{GN}}:=\inf_{f\neq 0} \frac{\|\partial_x^2u\|_{L^\infty(\R)}}{\|\partial_xu\|_{L^\infty(\R)}^{\frac{1}{3}}\|\partial_x^3u\|
	_{L^2(\R)}^{\frac{2}{3}}},\\
&C_{\mathrm{Sob}}:=\inf_{f\neq 0} \frac{\|f\|_{L^\infty(\R)}}{\|f\|_{H^1(\R)}},\quad
C_{\mathrm{Mor}}:=\inf_{f\neq 0}\frac{|f|_{C^{0,\frac{1}{2}}(\R)}}{\|f_x\|_{L^2(\R)}},
\end{aligned}
\end{equation*} 	
where the semi norm \(|\cdot|_{C^{0,\gamma}(\R)}\) is defined by
\begin{equation*}
\begin{aligned}
|f|_{C^{0,\gamma}(\R)}=\colon\sup_{x,y\in\R,\ x\neq y}\frac{|f(x)-f(y)|}{|x-y|^\gamma}.
\end{aligned}
\end{equation*} 

 We say that the solution of \eqref{eq:main-1} (\eqref{eq:BH} or \eqref{Whit} or \eqref{Whitresc}) exhibits wave breaking 
if there exists some \(T>0\) such that
\begin{equation*}
\begin{aligned}
|u(x,t)|<\infty,\quad x\in\R,\ t\in[0,T),
\end{aligned}
\end{equation*}	
but	
\begin{equation*}
\begin{aligned}
\inf_{\R}\partial_xu(x,t)\longrightarrow -\infty,\quad \text{as}\ t\longrightarrow T-.
\end{aligned}
\end{equation*}

Our first result can be stated precisely as follows:
\begin{theorem}[Burgers-Hilbert equation]\label{th:BH} Let \(\delta\in(0,1-\frac{\sqrt{3}}{2}]\). Assume that \(\phi\in H^2(\R)\) satisfies
	\begin{align}
	&\delta^2(\inf_{\R}\phi^\prime(x))^2> C_{\mathrm{Sob}}\|\phi\|_{H^2}+4\|\phi^{\prime}\|_{L^2}+64C_{\mathrm{Mor}}\|\phi^{\prime\prime}\|_{L^2},\label{t1c1}\\
	&-(1-\delta)^2\inf_{\R}\phi^\prime(x)> 6\bigg(\frac{\|\phi\|_{L^2}}{C_0}\bigg)+ 24C_{\mathrm{Mor}}\bigg(\frac{\|\phi^{\prime}\|_{L^2}}{C_0}\bigg),\label{t1c2}\\
	&-(1-\delta)^3\inf_{\R}\phi^\prime(x)> 8\bigg(\frac{\|\phi^{\prime}\|_{L^2}}{C_1}\bigg)+ 128C_{\mathrm{Mor}}\bigg(\frac{\|\phi^{\prime\prime}\|_{L^2}}{C_1}\bigg),\label{t1c3}
	\end{align}
	where the constant \(C_0\) and \(C_1\) satisfy	
	\begin{equation}\label{t1c4}
	\begin{aligned}
	\|\phi\|_{L^\infty}\leq\frac{C_0}{2},\quad 
	\|\phi^\prime\|_{L^\infty}\leq\frac{C_1}{2}.
	\end{aligned}
	\end{equation}	
	Then the solution to the Cauchy problem \eqref{eq:BH} with the initial data \(u(0,x)=\phi(x)\) exhibits wave breaking at some time \(T>0\). Moreover
	\begin{equation}\label{wb}
	\begin{aligned}
	-\frac{1}{\inf_{\R}\phi^\prime(x)}\frac{1}{1+\delta}<T<-\frac{1}{\inf_{\R}\phi^\prime(x)}\frac{1}{(1-\delta)^2}.
	\end{aligned}
	\end{equation}
	
\end{theorem}

\bigskip
In order to deal with the Whitham equation  we first collect the following property of \(K(x)\) \cite{Hur}:

\vspace{0.3cm}

\begin{lemma}\label{le:Hur} There exist constants \(L_0, L_\infty>0\) such that
	\begin{equation*}
	\begin{aligned}
	K(x)\leq \frac{L_0}{\sqrt{|x|}}\quad   \mathrm{and}\ |K^\prime(x)|\leq \frac{L_0}{\sqrt{|x|^3}},\quad  \mathrm{for}\   0<|x|\leq 1,
	\end{aligned}
	\end{equation*}
	and
	\begin{equation*}
	\begin{aligned}
	\int_1^\infty |K^\prime(x)|\,\diff x\leq L_\infty.
	\end{aligned}
	\end{equation*}	
\end{lemma}

Our result on the Whitham equation is as follows:
\begin{theorem}[Whitham equation]\label{th:W} Let \(\delta\in(0,1-\frac{2\sqrt{2}}{3}]\). Assume that \(\phi\in H^3(\R)\) satisfies
	\begin{align}
	&\delta^2(\inf_{\R}\phi^\prime(x))^2> 4L_0C_{\mathrm{Sob}}\|\phi\|_{H^3}+2C_1(3L_0+L_\infty)+36L_0C_{\mathrm{GN}}C_1^{\frac{1}{3}}\|\phi^{\prime\prime\prime}\|_{L^2}^{\frac{2}{3}},\label{t2c1}\\
	&-(1-\delta)^2\inf_{\R}\phi^\prime(x)> 8(3L_0+L_\infty)+16L_0\bigg(\frac{C_1}{C_0}\bigg),\label{t2c2}\\
	&-(1-\delta)^3\inf_{\R}\phi^\prime(x)> 4(3L_0+L_\infty)+36L_0C_{\mathrm{GN}}\bigg(\frac{\|\phi^{\prime\prime\prime}\|_{L^2}}{C_1}\bigg)^{\frac{2}{3}},\label{t2c3}
	\end{align}
	where the constant \(C_0\) and \(C_1\) satisfy 	
	\begin{equation}\label{t2c4}
	\begin{aligned}
	\|\phi\|_{L^\infty}\leq \frac{C_0}{2},\quad 
	\|\phi^\prime\|_{L^\infty}\leq\frac{C_1}{2}.
	\end{aligned}
	\end{equation}	
	Then the solution of the Cauchy problem \eqref{Whit} with  initial data \(u(0,x)=\phi(x)\) exhibits wave breaking at some time \(T>0\). Moreover
	\begin{equation*}
	\begin{aligned}
	-\frac{1}{\inf_{\R}\phi^\prime(x)}\frac{1}{1+\delta}<T<-\frac{1}{\inf_{\R}\phi^\prime(x)}\frac{1}{(1-\delta)^2}.
	\end{aligned}
	\end{equation*}
	
\end{theorem}

\begin{theorem}[fKdV equation: \(\alpha\in(-1,-\frac{2}{5})\)]\label{th:fKdV} Let \(\delta>0\) be sufficiently small and \(\alpha\in\big(-1,\frac{5(1-\delta)^2-7}{7-2(1-\delta)^2}\big)\). Assume that \(\phi\in H^3(\R)\) satisfies
	\begin{align}
	&\delta^2(\inf_{\R}\phi^\prime(x))^2> C_{\mathrm{Sob}}\|\phi\|_{H^2}+\frac{4C_1}{1+\alpha}+\frac{18C_{\mathrm{GN}}}{-\alpha}\bigg(C_1^{\frac{1}{3}}\|\phi^{\prime\prime\prime}\|_{L^2}^{\frac{2}{3}}\bigg),\label{t3c1}\\
	&-(1-\delta)^2\inf_{\R}\phi^\prime(x)> \frac{8}{-\alpha(1+\alpha)}+\frac{2}{\alpha^2 }\bigg(\frac{C_1}{C_0}\bigg),\label{t3c2}\\
	&-(1-\delta)^3\inf_{\R}\phi^\prime(x)> \frac{8}{1+\alpha}+\frac{36C_{\mathrm{GN}}}{-\alpha}\bigg(\frac{\|\phi^{\prime\prime\prime}\|_{L^2}}{C_1}\bigg)^{\frac{2}{3}},\label{t3c3}
	\end{align}
	where the constant \(C_0\) and \(C_1\) satisfy	
	\begin{equation}\label{t3c4}
	\begin{aligned}
	\|\phi\|_{L^\infty}\leq \frac{C_0}{2},\quad 
	\|\phi^\prime\|_{L^\infty}\leq \frac{C_1}{2}.
	\end{aligned}
	\end{equation}	
	Then the solution of the Cauchy problem \eqref{eq:main-1} with the initial data \(u(0,x)=\phi(x)\) exhibits wave breaking at some time \(T>0\). Moreover
	\begin{equation*}
	\begin{aligned}
	-\frac{1}{\inf_{\R}\phi^\prime(x)}\frac{1}{1+\delta}<T<-\frac{1}{\inf_{\R}\phi^\prime(x)}\frac{1}{(1-\delta)^2}.
	\end{aligned}
	\end{equation*}
	
\end{theorem}

		\begin{remark}
		\rm{It is easy to see that there exists some \(\phi\in H^2(\R)\) satisfying  \eqref{t1c1}-\eqref{t1c3} and \eqref{t1c4} in Theorem \ref{th:BH}. Indeed, given any \(\phi_0\in H^2(\R)\) with \(\inf_{\R}\phi_0^\prime(x)<0\), let \(\phi=\lambda\phi_0\)
			with \(\lambda>0\) and \(C_0=2\lambda\|\phi_0\|_{L^\infty},C_1=2\lambda\|\phi_0^\prime\|_{L^\infty}\) that obviously satisfy \eqref{t1c4}, then
			choosing \(\lambda\) sufficiently large, one checks that \(\phi\) satisfies \eqref{t1c1}-\eqref{t1c3} by comparing the powers of $\lambda$ in  both sides of each inequality. One can similarly analyze the assumptions in Theorem \ref{th:W} and \ref{th:fKdV}.}
		\end{remark}

\begin{remark}
	\rm{The result in Theorem \ref{th:BH} (Theorem \ref{th:fKdV}) does not contradict the existence of smooth solutions of the Burgers-Hilbert equation (fKdV equation) with initial data of size $\bigO(\epsilon)$ on the enhanced time scale $\bigO(1/\epsilon^2)$ which has been established in \cite{HI, HITW} (\cite{EW}).} 
\end{remark}

We finally give a very simple blowup result on the Burgers-Hilbert equation which reads as:
\begin{theorem}\label{th:main-2}  Assume that \(\phi\in H^2(\R)\) satisfies
	\begin{align}\label{b1}
	F(0)=\colon-\int_0^\infty \big(\phi(x)-\phi(0)\big)\exp(-x)\, \diff x\geq
	4\|\phi\|_{L^2}^{\frac{1}{2}}.
	\end{align}
	Then,  the lifespan \(T^*\) of the solution \(u\in C\big([0,T^*);H^2(\R)\big)\) to the Cauchy problem \eqref{eq:BH} with the initial data \(u(0,x)=\phi(x)\) is bounded above by
	\begin{align}\label{b2}
	T^*\leq\frac{4}{F(0)}=\colon T^{**},
	\end{align}
and
    \begin{align}\label{b2.5}
    \lim_{t\rightarrow T^{**}-}\|u_x\|_{L^\infty}=\infty.
    \end{align}
\end{theorem}
 
\begin{remark}
\rm{One can relax the assumption \(\phi\in H^2(\R)\) in Theorem \ref{th:main-2} as \(\phi\in L^2(\R)\cap C^{0,1}(\R)\) if there exists a solution in \(u\in L^2(\R)\cap C\big([0,T^*);C^{0,1}(\R)\big)\).}
\end{remark}
 
\begin{remark}\rm{To prove Theorem \ref{th:main-2}, we use a functional (see \eqref{half line}) with a smooth, fast decay weight defined on {\it half line} inspired by \cite{LW} which uses a similar functional to study the blowup of Euler and Euler-Poisson equations. The choice of this functional makes our proof very simple.
One may refer to \cite{CCG,Hur1} for the use of functionals with singular weights on the whole line to prove blowup of dispersive equations.}

 \end{remark}

 \section{Proof of Theorem \ref{th:BH}}\label{sec3} 
 \begin{proof}[Proof of Theorem \ref{th:BH}] 
 	
 It is standard that the Cauchy problem of \eqref{eq:BH} with \(u(0,x)=\phi(x)\) is well-posed in the class \(C\big([0,T):H^2(\R)\big)\) for some \(T>0\). We assume that \(T\) is the maximal time of existence hereafter.
 We define the particle path
 \begin{equation*}
 \begin{aligned}
 \frac{\diff }{\diff t}X(t,x)=u(X(t,x),t),\quad X(0,x)=x.
 \end{aligned}
 \end{equation*}
 Since \(u(x,t)\in C\big([0,T):H^2(\R)\big)\), the ODE theory shows that \(X(\cdot;x)\) exists throughout the interval \(t\in[0,T)\) for all \(x\in\R\).
 We denote
  \begin{equation*}
  \begin{aligned}
  v_0(t,x)=u(X(t,x),t),\quad v_1(t,x)=\partial_xu(X(t,x),t),
  \end{aligned}
  \end{equation*}
and 
  \begin{equation}\label{1}
  \begin{aligned}
  m(t)=\inf_{x\in\R}v_1(t;x)=\inf_{x\in\R}\partial_xu(x,t)=:m(0)q^{-1}(t).
  \end{aligned}
  \end{equation}
It is easy to see that
  \begin{align}
  &m(t)<0,\quad t\in [0,T),\label{2}\\
  &q(0)=1,\quad 
  q(t)>0,\quad t\in [0,T).\label{3}
  \end{align}
 It follows from \eqref{eq:BH} that 
 \begin{align}
 &\frac{\diff v_0}{\diff t}+K_0(t,x)=0,\label{4}\\ 
 &\frac{\diff v_1}{\diff t}+v_1^2+K_1(t,x)=0,\label{5}
 \end{align}
where 
 \begin{equation}\label{BH-kernal}
 \begin{aligned}
 &K_0(t,x)=\mathcal{H}u(X(t,x),t)
 =\int_\R \frac{\mathrm{sgn}(y)}{|y|}u(X(t,x)-y,t)\,\diff y,\\
 &K_1(t,x)=\mathcal{H}\partial_xu(X(t,x),t)
 =\int_\R \frac{\mathrm{sgn}(y)}{|y|}\partial_xu(X(t,x)-y,t)\,\diff y.
 \end{aligned}
 \end{equation}

The main ingredient in proving Theorem \ref{th:BH} is to show that
 \begin{equation}\label{8}
 \begin{aligned}
 |K_1(t,x)|<\delta^2 m^2(t),\quad \mathrm{for\ all}\ t\in[0,T)\ \mathrm{and} \ x\in\R.
 \end{aligned}
 \end{equation}
 Once \eqref{8} is shown, we may easily finish the proof. Indeed,
 for \(t \in [0, T)\), and any \(x\in \Sigma_{\delta}(t)=\Sigma_{\delta,1}(t)=\{x\in\R:v_1(t,x)\leq (1-\delta)m(t)\}\), we deduce by applying  Lemma \ref{le:a1} from Appendix with $t_1=0, t_2=t$ that
 \begin{equation*}
 \begin{aligned}
 m(0)\leq v_1(0,x)\leq (1-\delta)m(0),
 \end{aligned}
 \end{equation*}
and then combining this with \eqref{a3} and \eqref{a6.5} one sees that
 	\begin{equation*}
 	\begin{aligned}
 r(t,x)\leq m(0)(v_1^{-1}(0,x)+(1-\delta)t)\leq \frac{1}{1-\delta}+m(0)(1-\delta)t.
 	\end{aligned}
 	\end{equation*}
and
 	\begin{equation*}
 	\begin{aligned}
 	r(t,x)\geq m(0)(v_1^{-1}(0,x)+(1+\delta)t)\geq (1-\delta)+m(0)(1-\delta^2)t.
 	\end{aligned}
 	\end{equation*}
These two inequalities together with  \eqref{a4} give
 	\begin{equation*}
 	\begin{aligned}
 	(1-\delta)+m(0)(1-\delta^2)t\leq q(t)\leq \frac{1}{1-\delta}+m(0)(1-\delta)t,
 	\end{aligned}
 	\end{equation*}
that is
\begin{equation*}
\begin{aligned}
(1-\delta)+ \inf_{x\in\R}\phi^\prime(x)(1-\delta^2)t\leq q(t)\leq \frac{1}{1-\delta}+\inf_{x\in\R}\phi^\prime(x)(1-\delta)t.
\end{aligned}
\end{equation*}
Hence \eqref{wb} follows.
 	
In the rest	of this section, we turn to prove \eqref{8}. 
First observe that \eqref{8} holds at \(t=0\): 
 \begin{equation*}
 \begin{aligned}
 |K_1(0,x)|=|\mathcal{H}\phi^\prime(x)|\leq C_{\mathrm{Sob}}\|\phi\|_{H^2}<\delta^2 m^2(0),\quad  x\in\R,
 \end{aligned}
 \end{equation*}
where we have used the Sobolev embedding and the assumption \eqref{t1c1}.
We now prove \eqref{8} by contradiction for \(t\neq 0\). Suppose that \(|K_1(T_1,x_0)|=\delta^2 m^2(T_1)\) for some \(T_1\in(0,T)\) and some \(x_0\in\R\). 
 By continuity, without loss of generality, we may assume that
 \begin{equation}\label{9}
 \begin{aligned}
 |K_1(t,x)|\leq \delta^2 m^2(t),\quad \text{for\ all}\ t\in[0,T_1]\ \mathrm{and} \ x\in\R.
 \end{aligned}
 \end{equation}

 \bigskip
 	We claim that
 	\begin{align}\label{10}
 	\|v_0(t)\|_{L^\infty}=\|u(t)\|_{L^\infty}<C_0,\quad   \text{for\ all}\  t\in [0,T_1], 
 	\end{align}
 	and
 	\begin{align}\label{11}
 	\|v_1(t)\|_{L^\infty}=\|\partial_xu(t)\|_{L^\infty}<C_1q^{-1}(t),\quad   \text{for\ all}\  t\in [0,T_1], 
 	\end{align}
  where \(C_0,C_1\) satisfy \eqref{t1c4}. First observe by \eqref{t1c4} and \eqref{3} that 
 	\begin{equation*}
 	\begin{aligned}
 	\|v_0(0)\|_{L^\infty}=\|\phi\|_{L^\infty}<C_0,
 	\end{aligned}
 	\end{equation*}
 	and
 	\begin{equation*}
 	\begin{aligned}
 	\|v_1(0)\|_{L^\infty}=\|\phi^\prime\|_{L^\infty}<C_1q^{-1}(0).
 	\end{aligned}
 	\end{equation*}
 We will use a contradiction argument to show \eqref{10} and \eqref{11}. 
 	Suppose that \eqref{10} and \eqref{11} hold for all  \(t\in [0, T_2)\), but fails for either \eqref{10} or \eqref{11} at \(t = T_2\) for some \(T_2\in (0, T_1]\). Hence, by continuity, it holds
 	\begin{align}\label{12}
 	\|v_0(t)\|_{L^\infty}=\|u(t)\|_{L^\infty}<C_0,\quad   \text{for\ all}\  t\in [0,T_2], 
 	\end{align}
 	and
 	\begin{align}\label{13}
 	\|v_1(t)\|_{L^\infty}=\|\partial_xu(t)\|_{L^\infty}<C_1q^{-1}(t),\quad   \text{for\ all}\  t\in [0,T_2].
 	\end{align}

     To bound \(K_0(t;x)\), 
 	 we split, for   \(\eta\in(0,1]\), the integral into two parts  as follows:
 	\begin{equation*}
 	\begin{aligned}
 	K_0(t,x)=\underbrace{\int_{|y|<\eta}\frac{\mathrm{sgn}(y)}{|y|}u(X(t,x)-y,t)\,\diff y}_{I_1}
 	+\underbrace{\int_{|y|\geq\eta}\frac{\mathrm{sgn}(y)}{|y|}u(X(t,x)-y,t)\,\diff y}_{I_2}.
 	\end{aligned}
 	\end{equation*}
 	The term \(I_2\) can be easily estimated as:  	
 	\begin{equation}\label{14}
 	\begin{aligned}
 	|I_2|
 	\leq \bigg(\int_{|y|\geq\eta} \frac{1}{|y|^2}\,\diff y\bigg)^{1/2}\|u\|_{L^2}\leq 2\eta^{-\frac{1}{2}}\|\phi\|_{L^2},
 	\end{aligned}
 	\end{equation}
 	due to the conservation of \(\|u\|_{L^2}\).
 	For the term \(I_1\), we estimate
 	\begin{equation}\label{15}
 	\begin{aligned}
 	|I_1|
 	&=\bigg|\int_{|y|<\eta}\frac{\mathrm{sgn}(y)}{|y|}[u(X(t,x)-y,t)-u(X(t,x),t)]\,\diff y\bigg|\\
 	&\leq |u|_{C^{0,\frac{1}{2}}(\R)} \int_{|y|<\eta}\frac{1}{|y|}|y|^{\frac{1}{2}}\,\diff y
 	\leq 4C_{\mathrm{Mor}}\eta^{\frac{1}{2}}\|\partial_xu\|_{L^2(\R)},
 	\end{aligned}
 	\end{equation}
 	where we have used Morrey's inequality 
 		\begin{equation*}
 		\begin{aligned}
 		|u|_{C^{0,\frac{1}{2}}(\R)}\leq
 		C_{\mathrm{Mor}}\|\partial_xu\|_{L^2(\R)}.
 		\end{aligned}
 		\end{equation*}

It remains to control \(\|\partial_xu\|_{L^2}\).
Taking the first derivative \(\partial_x\) on \eqref{eq:BH} with respect to \(x\), multiplying it by \(\partial_xu\) and integrating it on \(x\) over \(\R\), one finally gets
\begin{equation*}
\begin{aligned}
\frac{1}{2}\frac{\diff}{\diff t}\int_\R(\partial_xu)^2\,\diff x
&=-\int_\R\partial_xu\mathcal{H}\partial_xu\,\diff x-
\int_\R [(\partial_xu)^3+u\partial_x^2u\partial_xu]\,\diff x\\
&=-\frac{1}{2}\int_\R (\partial_xu)^3\,\diff x,
\end{aligned}
\end{equation*}
where on the right hand side we have used the fact that the first term vanishes due to the anti-symmetry of \(\mathcal{H}\) and integration by parts in the second term.
Hence, one obtains
\begin{equation*}
\begin{aligned}
\frac{\diff}{\diff t}\int_\R(\partial_xu)^2\,\diff x
&=-\int_\R\partial_xu(\partial_xu)^2\,\diff x
\leq - m(t)\int_\R(\partial_xu)^2\,\diff x\\
&=- m(0)q^{-1}(t)\int_\R(\partial_xu)^2\,\diff x,
\end{aligned}
\end{equation*}
which combines with \eqref{a8} implies for all \(t\in [0, T_2 ]\) that 
\begin{equation}\label{16}
\begin{aligned}
\|\partial_xu(t)\|_{L^2}
&\leq \|\phi^{\prime}\|_{L^2}(1-\delta)^{-\frac{1}{2(1-\delta)^2}}
q(t)^{-\frac{1}{2(1-\delta)^2}}\\
&\leq 2\|\phi^{\prime}\|_{L^2}
q(t)^{-\frac{1}{2(1-\delta)^2}},
\end{aligned}
\end{equation}
where we have used \(\delta\in(0,1-\frac{\sqrt{3}}{2}]\).
One finally obtains from \eqref{15} and \eqref{16} that
\begin{equation}\label{17}
\begin{aligned}
|I_1|
\leq 8C_{\mathrm{Mor}}\|\phi^{\prime}\|_{L^2}\eta^{\frac{1}{2}}
q(t)^{-\frac{1}{2(1-\delta)^2}}.
\end{aligned}
\end{equation}

 	By choosing \(\eta=q(t)^{\frac{1}{2(1-\delta)^2}}\), for all \(t\in[0,T_2]\) and \(x\in\R\), we see from \eqref{14} and \eqref{17} that
 	\begin{equation}\label{K_0}
 	\begin{aligned}
 	|K_0(t,x)|&\leq 8C_{\mathrm{Mor}}\|\phi^{\prime}\|_{L^2}\eta^{\frac{1}{2}}
 	q(t)^{-\frac{1}{2(1-\delta)^2}}+2\eta^{-\frac{1}{2}}\|\phi\|_{L^2}\\
 	&\leq (2\|\phi\|_{L^2}+8C_{\mathrm{Mor}}\|\phi^{\prime}\|_{L^2})q(t)^{-\frac{1}{4(1-\delta)^2}}\\
 	&\leq (2\|\phi\|_{L^2}+8C_{\mathrm{Mor}}\|\phi^{\prime}\|_{L^2})q(t)^{-\frac{1}{3}},
 	\end{aligned}
 	\end{equation}
 where we have used \(\delta\in(0,1-\frac{\sqrt{3}}{2}]\) and \eqref{a5}. 
 In view of \eqref{4}, \eqref{K_0} and \eqref{a7}, we may now control \(v_0(t,x)\) for all \(t\in[0,T_2]\) and for all \(x\in\R\) as follows:
 	\begin{equation}\label{18}
 	\begin{aligned}
 	&\quad|v_0(t,x)|\leq \|\phi\|_{L^\infty}+\int_0^{t}|K_0(\tau,x)|\,\diff \tau\\
 	&\leq \frac{1}{2}C_0+(2\|\phi\|_{L^2}+8C_{\mathrm{Mor}}\|\phi^{\prime}\|_{L^2})\int_0^{t}
 	q^{-\frac{1}{3}}(\tau)\,\diff \tau\\
 	&\leq \frac{1}{2}C_0-(3\|\phi\|_{L^2}+12C_{\mathrm{Mor}}\|\phi^{\prime}\|_{L^2})m^{-1}(0)(1-\delta)^{-\frac{4}{3}}[(1-\delta)^{-\frac{2}{3}}-q^{\frac{2}{3}}(t)]\\
 	&\leq \frac{1}{2}C_0-(3\|\phi\|_{L^2}+12C_{\mathrm{Mor}}\|\phi^{\prime}\|_{L^2})(1-\delta)^{-2}m^{-1}(0)\\
 	&< C_0,
 	\end{aligned}
 	\end{equation}
where we have used 
\begin{equation*}
\begin{aligned}
\|\phi\|_{L^\infty}\leq \frac{1}{2}C_0\leq \frac{1}{2}C_0q^{-1}(t), 
\end{aligned}
\end{equation*}
due to the assumption \eqref{t1c4} and \eqref{a5} in the first inequality, and \eqref{t1c2} in the last inequality.

 	To estimate \(K_1(t;x)\), we again, for   \(\eta\in(0,1]\), split the integral into two parts as follows:
 	\begin{equation*}
 	\begin{aligned}
 	K_1(t,x)=\underbrace{\int_{|y|<\eta}\frac{\mathrm{sgn}(y)}{|y|}\partial_xu(X(t,x)-y,t)\,\diff y}_{I_3}
 	+\underbrace{\int_{|y|\geq\eta}\frac{\mathrm{sgn}(y)}{|y|}\partial_xu(X(t,x)-y,t)\,\diff y}_{I_4}.
 	\end{aligned}
 	\end{equation*}	
 	With  the same manipulation as in  \(I_2\), one can estimate  	
 	 	\begin{equation}\label{19}
 	 	\begin{aligned}
 	 	|I_4|
 	 	\leq 2\eta^{-\frac{1}{2}}\|\partial_xu\|_{L^2}\leq 4\|\phi^{\prime}\|_{L^2}\eta^{-\frac{1}{2}}
 	 	q(t)^{-\frac{1}{2(1-\delta)^2}},
 	 	\end{aligned}
 	 	\end{equation}
 where we have invoked \eqref{16}.
 In a similar fashion to \(I_1\), we have		
	\begin{equation}\label{20}
	\begin{aligned}
	|I_3|
	\leq 4C_{\mathrm{Mor}}\eta^{\frac{1}{2}}\|\partial_x^2u\|_{L^2(\R)}.
	\end{aligned}
	\end{equation}	
		
It remains to control \(\|\partial_x^2u\|_{L^2}\). Similarly to treatment of  the first derivative  \(\partial_xu\), we have
 \begin{equation*}
 \begin{aligned}
 \frac{\diff}{\diff t}\int_\R(\partial_x^2u)^2\,\diff x
 &=-5\int_\R\partial_xu(\partial_x^2u)^2\,\diff x
 \leq -5 m(t)\int_\R(\partial_x^2u)^2\,\diff x\\
 &=-5 m(0)q^{-1}(t)\int_\R(\partial_x^2u)^2\,\diff x,
 \end{aligned}
 \end{equation*}
 which gives for all \(t\in [0, T_2 ]\) that
 \begin{equation}\label{21}
 \begin{aligned}
 \|\partial_x^2u(t)\|_{L^2}
 &\leq \|\phi^{\prime\prime}\|_{L^2}(1-\delta)^{-\frac{5}{2(1-\delta)^2}}
 q(t)^{-\frac{5}{2(1-\delta)^2}}\\
 &\leq 16 \|\phi^{\prime\prime}\|_{L^2}
 q(t)^{-\frac{5}{2(1-\delta)^2}},
 \end{aligned}
 \end{equation}
where we have used \(\delta\in(0,1-\frac{\sqrt{3}}{2}]\). 
 It follows from \eqref{20} and \eqref{21}
 	\begin{equation}\label{22}
 	\begin{aligned}
 	|I_3|
 	\leq 64C_{\mathrm{Mor}}\|\phi^{\prime\prime}\|_{L^2}\eta^{\frac{1}{2}}
 	q(t)^{-\frac{5}{2(1-\delta)^2}}.
 	\end{aligned}
 	\end{equation}

 	In view of \eqref{19} and \eqref{22}, taking \(\eta=q(t)^{\frac{2}{(1-\delta)^2}}\), we conclude for all \(t\in [0, T_2 ]\) and \(x\in\R\) that	
 	\begin{equation}\label{23}
 	\begin{aligned}
 	|K_1(t,x)|&\leq 64C_{\mathrm{Mor}}\|\phi^{\prime\prime}\|_{L^2}\eta^{\frac{1}{2}}
 	q(t)^{-\frac{5}{2(1-\delta)^2}}+4\|\phi^{\prime}\|_{L^2}\eta^{-\frac{1}{2}}
 	q(t)^{-\frac{1}{2(1-\delta)^2}}\\
 	&\leq (4\|\phi^{\prime}\|_{L^2}+64C_{\mathrm{Mor}}\|\phi^{\prime\prime}\|_{L^2})q(t)^{-\frac{3}{2(1-\delta)^2}}\\
 	&\leq (4\|\phi^{\prime}\|_{L^2}+64C_{\mathrm{Mor}}\|\phi^{\prime\prime}\|_{L^2})q(t)^{-2},
 	\end{aligned}
 	\end{equation}
 	where we have used \eqref{a5} and
 	\begin{equation*}
 	\begin{aligned}
 	-\frac{3}{2(1-\delta)^2}\geq -2,
 	\end{aligned}
 	\end{equation*}
 which follows from 
 	 \(\delta\in(0,1-\frac{\sqrt{3}}{2}]\).	Recalling \eqref{5} that
 	\begin{equation*}
 	\begin{aligned}
 	\frac{\diff v_1}{\diff t}=-v_1^2-K_1(t,x)\leq |K_1(t,x)|,
 	\end{aligned}
 	\end{equation*}
one uses \eqref{23} and \eqref{a7} to estimate for all \(t\in [0, T_2 ]\) and \(x\in\R\) that
 	\begin{equation}\label{24}
 	\begin{aligned}
 	&\quad v_1(t,x)
 	\leq\|\phi^\prime\|_{L^\infty}+(4\|\phi^{\prime}\|_{L^2}+64C_{\mathrm{Mor}}\|\phi^{\prime\prime}\|_{L^2})\int_0^tq^{-2}(\tau)\,\diff \tau\\
 	&\leq\frac{1}{2}C_1q^{-1}(t)-(1-\delta)^{-3}m^{-1}(0)(4\|\phi^{\prime}\|_{L^2}+64C_{\mathrm{Mor}}\|\phi^{\prime\prime}\|_{L^2})[q^{-1}(t)-(1-\delta)]\\
 	&\leq\frac{1}{2}C_1q^{-1}(t)-(1-\delta)^{-3}m^{-1}(0)(4\|\phi^{\prime}\|_{L^2}+64C_{\mathrm{Mor}}\|\phi^{\prime\prime}\|_{L^2})q^{-1}(t)\\
 	&<C_1q^{-1}(t),
 	\end{aligned}
 	\end{equation}
 	where we have used 
 	\begin{equation*}
 	\begin{aligned}
 	\|\phi^\prime\|_{L^\infty}\leq \frac{1}{2}C_1\leq \frac{1}{2}C_1q^{-1}(t), 
 	\end{aligned}
 	\end{equation*}
 	due to the assumption \eqref{t1c4} and \eqref{a5} in the first inequality, and the assumption \eqref{t1c3} in the last inequality.
 	On the other hand,  \eqref{t1c4} and \eqref{1} imply that
 	\begin{equation}\label{25}
 	\begin{aligned}
 	v_1(t,x)\geq m(t)=m(0)q^{-1}(t)
 	\geq-\frac{1}{2}C_1q^{-1}(t),
 	\end{aligned}
 	\end{equation}
 for all \(t\in [0, T_2 ]\) and \(x\in\R\). 
 
 A contradiction to \eqref{12}-\eqref{13} occurs following from \eqref{18}, \eqref{24} and \eqref{25}. 
 Now we go back to \eqref{23} and use \eqref{t1c1} to find that
 	\begin{equation*}
 	\begin{aligned}
 	|K_1(t,x)|\leq (4\|\phi^{\prime}\|_{L^2}+64C_{\mathrm{Mor}}\|\phi^{\prime\prime}\|_{L^2})m^{-2}(0)m^2(t)<\delta^2m^2(t),
 	\end{aligned}
 	\end{equation*}
 	for all \(t\in [0, T_1 ]\) and all \(x\in\R\). We get a contradiction to \eqref{9}! This means we have shown \eqref{8} for all \(t\in [0, T)\) and all \(x\in\R\).

\end{proof}

\section{Proof of Theorem \ref{th:W}}\label{sec4}  
\begin{proof}[Proof of Theorem \ref{th:W}]
	We first note that the Cauchy problem of \eqref{Whit} with \(u(0,x)=\phi(x)\) is well-posed in the class \(C\big([0,T):H^3(\R)\big)\) for some \(T>0\) and we now assume that \(T\) is the maximal time of existence. Using the same notations \(X(t,x),v_0(t,x),v_1(t,x),m(t)\) and \(q(t)\) as Section \ref{sec3}, it then 
follows from \eqref{Whit} that 
	\begin{align}
	&\frac{\diff v_0}{\diff t}+K_0(t,x)=0,\label{29}\\ 
	&\frac{\diff v_1}{\diff t}+v_1^2+K_1(t,x)=0,\label{30}
	\end{align}
	where 	 
	\begin{equation}\label{Whit-kernal}
	\begin{aligned}
	K_0(t,x)=\int_\R K(y)\partial_xu(X(t,x)-y,t)\,\diff y,\\
	K_1(t,x)=\int_\R K(y)\partial_x^2u(X(t,x)-y,t)\,\diff y.
	\end{aligned}
	\end{equation}

To prove Theorem \ref{th:W}, it suffices to show that
	\begin{equation}\label{31}
	\begin{aligned}
	|K_1(t,x)|<\delta^2 m^2(t),\quad \mathrm{for\ all}\ t\in[0,T)\ \mathrm{and} \ x\in\R.
	\end{aligned}
	\end{equation}
We first check that \eqref{31} holds at \(t=0\). To estimate \(K_1(0;x)\), we split the integral as follows: 
	\begin{equation*}
	\begin{aligned}
	&K_1(0,x)=\int_\R K(y)\phi^{\prime\prime}(x-y)\,\diff y\\
	&=\int_{|y|< 1} K(y)\phi^{\prime\prime}(x-y)\,\diff y+\int_{|y|\geq 1} K(y)\phi^{\prime\prime}(x-y)\,\diff y.
	\end{aligned}
	\end{equation*}
   In view of Lemma \ref{le:Hur}, one has 	
	\begin{equation}\label{32}
	\begin{aligned}
	\bigg|\int_{|y|< 1} K(y)\phi^{\prime\prime}(x-y)\,\diff y\bigg|
	\leq  \|\phi^{\prime\prime}\|_{L^\infty}\bigg|\int_{|y|< 1} K(y)\,\diff y\bigg|
	\leq  4L_0C_{\mathrm{Sob}}\|\phi\|_{H^3}.
	\end{aligned}
	\end{equation}
	where we have used Sobolev embedding.
	We use integration by parts to get 	
	\begin{equation}\label{33}
	\begin{aligned}
	&\bigg|\int_{|y|\geq 1} K(y)\phi^{\prime\prime}(x-y)\,\diff y\bigg|\\
	&\leq \big|K(1)[\phi^{\prime}(-1-y)-\phi^{\prime}(1-y)]\big|
	+\bigg|\int_{|y|\geq 1} K^{\prime}(y)\phi^{\prime}(x-y)\,\diff y\bigg|\\
	&\leq 2L_0\|\phi^{\prime}\|_{L^\infty}+\|\phi^{\prime}\|_{L^\infty}\bigg|\int_{|y|\geq 1} K^{\prime}(y)\,\diff y\bigg|\\
	&\leq 2(L_0+L_\infty)\|\phi^{\prime}\|_{L^\infty}\leq C_1(L_0+L_\infty),
	\end{aligned}
	\end{equation}
	where we have used Lemma \ref{le:Hur} in the third inequality and \eqref{t2c4} in the last inequality.  	
	It follows from \eqref{32} and \eqref{33} that 	 
	\begin{equation*}
	\begin{aligned}
	|K_1(0,x)|\leq 4L_0C_{\mathrm{Sob}}\|\phi\|_{H^3}+C_1(L_0+L_\infty)<\delta^2 m^2(0),\quad  x\in\R,
	\end{aligned}
	\end{equation*}
	where we have used \eqref{t2c1}.
	
	We now turn to prove \eqref{31} by contradiction for \(t\neq 0\). Suppose that \(|K_1(T_1,x_0)|=\delta^2 m^2(T_1)\) for some \(T_1\in(0,T)\) and some \(x_0\in\R\). 
	By continuity, without loss of generality, we may assume that
	\begin{equation}\label{34}
	\begin{aligned}
	|K_1(t,x)|\leq \delta^2 m^2(t),\quad \text{for\ all}\ t\in[0,T_1]\ \mathrm{and} \ x\in\R.
	\end{aligned}
	\end{equation}
We claim that
\begin{align}\label{35}
\|v_0(t)\|_{L^\infty}=\|u(t)\|_{L^\infty}<C_0,\quad   \text{for\ all}\  t\in [0,T_1], 
\end{align}
and
\begin{align}\label{36}
\|v_1(t)\|_{L^\infty}=\|\partial_xu(t)\|_{L^\infty}<C_1q^{-1}(t),\quad   \text{for\ all}\  t\in [0,T_1], 
\end{align}
where \(C_0,C_1\) satisfy \eqref{t2c4}. First observe that 
\begin{equation*}
\begin{aligned}
\|v_0(0)\|_{L^\infty}=\|\phi\|_{L^\infty}<C_0,
\end{aligned}
\end{equation*}
and
\begin{equation*}
\begin{aligned}
\|v_1(0)\|_{L^\infty}=\|\phi^\prime\|_{L^\infty}<C_1q^{-1}(0).
\end{aligned}
\end{equation*}
A contradiction argument will be used to show \eqref{35} and \eqref{36}. 
Suppose that \eqref{35} and \eqref{36} hold for all  \(t\in [0, T_2)\), but fails for either \eqref{35} or \eqref{36} at \(t = T_2\) for some \(T_2\in (0, T_1]\). Hence, by continuity, it holds
\begin{align}\label{37}
\|v_0(t)\|_{L^\infty}=\|u(t)\|_{L^\infty}<C_0,\quad   \text{for\ all}\  t\in [0,T_2], 
\end{align}
and
\begin{align}\label{38}
\|v_1(t)\|_{L^\infty}=\|\partial_xu(t)\|_{L^\infty}<C_1q^{-1}(t),\quad   \text{for\ all}\  t\in [0,T_2].
\end{align}
	
	To control \(K_0(t,x)\), we split the integral with \(\eta\in(0,1]\) as follows:
	\begin{equation*}
	\begin{aligned}
	K_0(t,x)=\underbrace{\int_{|y|\leq\eta}K(y)\partial_xu(X(t,x)-y,t)\,\diff y}_{I_1}
+\underbrace{\int_{|y|>\eta}K(y)\partial_xu(X(t,x)-y,t)\,\diff y}_{I_2}.
	\end{aligned}
	\end{equation*}
	For the term \(I_1\), using Lemma \ref{le:Hur} and \eqref{38}, we have  	
	\begin{equation}\label{39}
	\begin{aligned}
	|I_1|
	\leq 2\int_{|y|\leq\eta} \frac{L_0}{\sqrt{|y|}}\,\diff y\cdot\|v_1\|_{L^\infty}\leq 4L_0\eta^{\frac{1}{2}}\|v_1\|_{L^\infty}\leq 4L_0C_1\eta^{\frac{1}{2}}q^{-1}(t).
	\end{aligned}
	\end{equation}
	Considering the term \(I_2\), we use integration by parts to bound it as follows:	
	\begin{equation}\label{40}
	\begin{aligned}
	|I_2|
	&\leq\big|K(\eta)[u(X(t,x)-\eta,t)-u(X(t,x)+\eta,t)]\big|\\
	&\quad+\bigg|\int_{\eta<|y|\leq 1} K^\prime(y)u(X(t,x)-y,t)\,\diff y\bigg|\\
	&\quad+\bigg|\int_{|y|>1} K^\prime(y)u(X(t,x)-y,t)\,\diff y\bigg|\\
	&\leq 2L_0\eta^{-\frac{1}{2}}\|v_0\|_{L^\infty}+4L_0(\eta^{-\frac{1}{2}}-1)\|v_0\|_{L^\infty}
	+2L_\infty\|v_0\|_{L^\infty}\\
	&\leq 2(3L_0\eta^{-\frac{1}{2}}+L_\infty)\|v_0\|_{L^\infty}
	\leq 2C_0(3L_0+L_\infty)\eta^{-\frac{1}{2}},
	\end{aligned}
	\end{equation}
	where we have used Lemma \ref{le:Hur} in the third inequality and \eqref{37} in the last inequality.

	In view of \eqref{39} and \eqref{40}, one chooses \(\eta=q(t)\) to get
	\begin{equation}\label{41}
	\begin{aligned}
	|K_0(t,x)|\leq 2\big[C_0(3L_0+L_\infty)+2L_0C_1\big]q^{-\frac{1}{2}}(t),
	\end{aligned}
	\end{equation}
	for all \(t\in[0,T_2]\) and for all \(x\in\R\).
	By \eqref{29}, \eqref{41} and \eqref{a7}, we may now control \(v_0(t;x)\) for all \(t\in[0,T_2]\) and for all \(x\in\R\) as follows:
	\begin{equation}\label{42}
	\begin{aligned}
	&\quad|v_0(t,x)|\leq \|\phi\|_{L^\infty}+\int_0^{t}|K_0(\tau,x)|\,\diff \tau\\
	&\leq \frac{1}{2}C_0+2\big[C_0(3L_0+L_\infty)+2L_0C_1\big]\int_0^{t}
	q^{-\frac{1}{2}}(\tau)\,\diff \tau\\
	&\leq \frac{1}{2}C_0-4\big[C_0(3L_0+L_\infty)+2L_0C_1\big]m^{-1}(0)(1-\delta)^{-\frac{3}{2}}[(1-\delta)^{-\frac{1}{2}}-q^{\frac{1}{2}}(t)]\\
	&\leq \frac{1}{2}C_0-4\big[C_0(3L_0+L_\infty)+2L_0C_1\big](1-\delta)^{-2}m^{-1}(0)\\
	&< C_0.
	\end{aligned}
	\end{equation}
	where we have used \eqref{t2c2} in the last inequality.

	To bound \(K_1(t,x)\), we proceed as:
	\begin{equation*}
	\begin{aligned}
	K_1(t,x)=\underbrace{\int_{|y|\leq\eta}K(y)\partial_x^2u(X(t,x)-y,t)\,\diff y}_{I_3}
	+\underbrace{\int_{|y|>\eta}K(y)\partial_x^2u(X(t,x)-y,t)\,\diff y}_{I_4}.
	\end{aligned}
	\end{equation*}
	Similar to \(I_2\), by integration by parts, one has	
	\begin{equation}\label{43}
	\begin{aligned}
	|I_4|
	&\leq 2L_0\eta^{-\frac{1}{2}}\|v_1\|_{L^\infty}+4L_0(\eta^{-\frac{1}{2}}-1)\|v_1\|_{L^\infty}+2L_\infty\|v_1\|_{L^\infty}\\
	&\leq 2C_1(3L_0+L_\infty)\eta^{-\frac{1}{2}}q^{-1}(t),
	\end{aligned}
	\end{equation}
	where we have used Lemma \ref{le:Hur} in the second inequality and \eqref{38} in the last inequality.	
	For the term \(I_3\), one estimates
	\begin{equation}\label{44}
	\begin{aligned}
	|I_3|
	\leq 2\int_{|y|\leq\eta} \frac{L_0}{\sqrt{|y|}}\,\diff y \cdot\|\partial_x^2u\|_{L^\infty}\leq 4L_0\eta^{\frac{1}{2}}\|\partial_x^2u\|_{L^\infty},
	\end{aligned}
	\end{equation}
	where we have used Lemma \ref{le:Hur} again. 
	
	In order to control \(\|\partial_x^2u\|_{L^\infty}\),
	we need to estimate \(\|\partial_x^3u\|_{L^2}\). 
	We differentiate \eqref{Whit} three times with respect to \(x\), multiply by \(\partial_x^3u\) and integrate  on \(x\) over \(\R\) to get 
	\begin{equation}\label{45}
	\begin{aligned}
	\frac{1}{2}\frac{\diff}{\diff t}\int_\R(\partial_x^3u)^2\,\diff x
	&=-\int_\R\partial_x^3u\int_\R K(x-y)\partial_y^4u(y)\,\diff y\,\diff x\\
	&\quad-
	\int_\R [4\partial_xu(\partial_x^3u)^2+3(\partial_x^2u)^2\partial_x^3u+u\partial_x^4u\partial_x^3u]\,\diff x.
	\end{aligned}
	\end{equation}
	Obviously, the first term on the right hand side of \eqref{45} vanishes since \(K(\cdot)\) is even. On the other hand, one uses integration by parts to see that
	\begin{equation}\label{46}
	\begin{aligned}
	&\int_\R(\partial_x^2u)^2\partial_x^3u\,\diff x=0,\\
	&\int_\R u\partial_x^4u\partial_x^3u\,\diff x=-\frac{1}{2}\int_\R\partial_xu(\partial_x^3u)^2\,\diff x.
	\end{aligned}
	\end{equation}
	We substitute \eqref{46} into \eqref{45} to deduce  
	\begin{equation}\label{46.5}
	\begin{aligned}
	\frac{\diff}{\diff t}\int_\R(\partial_x^3u)^2\,\diff x
	=-7\int_\R\partial_xu(\partial_x^3u)^2\,\diff x\leq -7 m(0)q^{-1}(t)\|\partial_x^3u\|_{L^2}^2.
	\end{aligned}
	\end{equation}
Solving \eqref{46.5} by using \eqref{a8} gives
	\begin{equation}\label{47}
	\begin{aligned}
	\|\partial_x^3u\|_{L^2}
	&\leq \|\phi^{\prime\prime\prime}\|_{L^2}(1-\delta)^{-\frac{7}{2(1-\delta)^2}}
	q(t)^{-\frac{7}{2(1-\delta)^2}}\\
	&\leq 16\|\phi^{\prime\prime\prime}\|_{L^2}
	q(t)^{-\frac{7}{2(1-\delta)^2}}
	\end{aligned}
	\end{equation}
for all \(t\in [0, T_2 ]\), where we have used \(\delta\in(0,1-\frac{2\sqrt{2}}{3}]\). However the bound \eqref{47} for \(\|\partial_x^3u\|_{L^2}\) is too bad to control \(\|\partial_x^2u\|_{L^\infty}\) by Sobolev embedding directly. To get a better bound, 
 we use Gagliardo-Nirenberg interpolation to deduce
\begin{equation}\label{47.5}
\begin{aligned}
\|\partial_x^2u\|_{L^\infty}\leq C_{\mathrm{GN}} \|\partial_xu\|_{L^\infty}^{\frac{1}{3}}\|\partial_x^3u\|_{L^2}^{\frac{2}{3}}\leq 9
C_{\mathrm{GN}}C_1^{\frac{1}{3}}\|\phi^{\prime\prime\prime}\|_{L^2}^{\frac{2}{3}}q(t)^{-\frac{1}{3}-\frac{7}{3(1-\delta)^2}},
\end{aligned}
\end{equation}
where we have used \eqref{38} and \eqref{47}.	
We then may estimate in view of \eqref{44} and \eqref{47.5} that
\begin{equation}\label{48}
\begin{aligned}
|I_3|
\leq 36L_0C_{\mathrm{GN}}C_1^{\frac{1}{3}}\|\phi^{\prime\prime\prime}\|_{L^2}^{\frac{2}{3}}\eta^{\frac{1}{2}}
q(t)^{-\frac{1}{3}-\frac{7}{3(1-\delta)^2}}.
\end{aligned}
\end{equation}

	In light of \eqref{43} and \eqref{48}, we take \(\eta=q(t)^{-\frac{2}{3}+\frac{7}{3(1-\delta)^2}}\)  to obtain	
	\begin{equation}\label{49}
	\begin{aligned}
	|K_1(t,x)|
	&\leq \big[2C_1(3L_0+L_\infty)+36L_0C_{\mathrm{GN}}C_1^{\frac{1}{3}}\|\phi^{\prime\prime\prime}\|_{L^2}^{\frac{2}{3}}\big]q(t)^{-\frac{2}{3}-\frac{7}{6(1-\delta)^2}}\\
	&\leq \big[2C_1(3L_0+L_\infty)+36L_0C_{\mathrm{GN}}C_1^{\frac{1}{3}}\|\phi^{\prime\prime\prime}\|_{L^2}^{\frac{2}{3}}\big]q^{-2}(t),
	\end{aligned}
	\end{equation}
	for all \(t\in[0,T_2]\) and for all \(x\in\R\), where we have used 
	\begin{equation*}
	\begin{aligned}
	-\frac{2}{3}-\frac{7}{6(1-\delta)^2}\geq -2,
	\end{aligned}
	\end{equation*}
	which follows from 	
	\(\delta\in(0,1-\frac{2\sqrt{2}}{3}]\) and \eqref{a5}.
	Recalling \eqref{30} that
	\begin{equation*}
	\begin{aligned}
	\frac{\diff v_1}{\diff t}=-v_1^2-K_1(t,x)\leq |K_1(t,x)|,
	\end{aligned}
	\end{equation*}
	one uses \eqref{49} and \eqref{a7} to estimate
	\begin{equation}\label{50}
	\begin{aligned}
	&v_1(t,x)\\
	&\leq\|\phi^\prime\|_{L^\infty}+\big[2C_1(3L_0+L_\infty)+36L_0C_{\mathrm{GN}}C_1^{\frac{1}{3}}\|\phi^{\prime\prime\prime}\|_{L^2}^{\frac{2}{3}}\big]\int_0^tq^{-2}(\tau)\,\diff \tau\\
	&\leq\frac{1}{2}C_1q^{-1}(t)-(1-\delta)^{-3}m^{-1}(0)\big[2C_1(3L_0+L_\infty)
	+36L_0C_{\mathrm{GN}}C_1^{\frac{1}{3}}\|\phi^{\prime\prime\prime}\|_{L^2}^{\frac{2}{3}}\big][q^{-1}(t)-(1-\delta)]\\
	&\leq\frac{1}{2}C_1q^{-1}(t)-(1-\delta)^{-3}m^{-1}(0)\big[2C_1(3L_0+L_\infty)+36L_0C_{\mathrm{GN}}C_1^{\frac{1}{3}}\|\phi^{\prime\prime\prime}\|_{L^2}^{\frac{2}{3}}\big]q^{-1}(t)\\
	&<C_1q^{-1}(t),
	\end{aligned}
	\end{equation}
	where we have used \eqref{t2c3} in the last inequality. 
	On the other hand, one has
	\begin{equation}\label{51}
	\begin{aligned}
	v_1(t,x)\geq m(t)=m(0)q^{-1}(t)
	\geq-\frac{1}{2}C_1q^{-1}(t).
	\end{aligned}
	\end{equation}
	for all \(t\in [0, T_2 ]\) and \(x\in\R\). 
	
	A contradiction to \eqref{37}-\eqref{38} occurs following from \eqref{42}, \eqref{50} and \eqref{51}. 
	Now we go back to \eqref{49} and use \eqref{t2c1} to find that
	\begin{equation*}
	\begin{aligned}
	|K_1(t,x)|\leq [2C_1(3L_0+L_\infty)+36L_0C_{\mathrm{GN}}C_1^{\frac{1}{3}}\|\phi^{\prime\prime\prime}\|_{L^2}^{\frac{2}{3}}\big]m^{-2}(0)m^2(t)
	<\delta^2m^2(t),
	\end{aligned}
	\end{equation*}
	for all \(t\in [0, T_1 ]\) and all \(x\in\R\). We get a contradiction to \eqref{34}! This means that we have shown \eqref{31} for all \(t\in [0, T)\) and all \(x\in\R\).

\end{proof}

\section{Proof of Theorem \ref{th:fKdV}}\label{sec5}
\begin{proof}[Proof of Theorem \ref{th:fKdV}] 
	
	We first note that the Cauchy problem of \eqref{eq:main-1} with \(u(0,x)=\phi(x)\) is well-posed in the class \(C\big([0,T):H^3(\R)\big)\) for some \(T>0\) and we now assume that \(T\) is the maximal time of existence. Using the same notations \(X(t,x),v_0(t,x),v_1(t,x),m(t)\) and \(q(t)\) as in Section \ref{sec3},
	it then follows from \eqref{eq:main-1} that 
	\begin{align}
	&\frac{\diff v_0}{\diff t}+K_0(t,x)=0,\label{54}\\ 
	&\frac{\diff v_1}{\diff t}+v_1^2+K_1(t,x)=0,\label{55}
	\end{align}
	where 
	\begin{equation}\label{fKdV-kernal}
	\begin{aligned}
	&K_0(t,x)
	=\int_\R\frac{\mathrm{sgn}(y)}{|y|^{2+\alpha}}[u(X(t,x),t)-u(X(t,x)-y,t)]\,\diff y,\\
	&K_1(t,x)
	=\int_\R \frac{\mathrm{sgn}(y)}{|y|^{2+\alpha}}[\partial_xu(X(t,x),t)-\partial_xu(X(t,x)-y,t)]\,\diff y.
	\end{aligned}
	\end{equation}	
	
	We are done if we show
	\begin{equation}\label{56}
	\begin{aligned}
	|K_1(t,x)|<\delta^2 m^2(t),\quad \mathrm{for\ all}\ t\in[0,T)\ \mathrm{and} \ x\in\R.
	\end{aligned}
	\end{equation}
	In view of \eqref{t3c1}, one easily checks that \eqref{56} holds at \(t=0\). 
	We will prove \eqref{56} by contradiction. Suppose that \(|K_1(T_1,x_0)|=\delta^2 m^2(T_1)\) for some \(T_1\in(0,T)\) and some \(x_0\in\R\). 
	By continuity, without loss of generality, we may assume that
	\begin{equation*}
	\begin{aligned}
	|K_1(t,x)|\leq \delta^2 m^2(t),\quad \text{for\ all}\ t\in[0,T_1]\ \mathrm{and} \ x\in\R.
	\end{aligned}
	\end{equation*}

	\bigskip
	We claim that
	\begin{align}\label{57}
	\|v_0(t)\|_{L^\infty}=\|u(t)\|_{L^\infty}<C_0,\quad   \text{for\ all}\  t\in [0,T_1], 
	\end{align}
	and
	\begin{align}\label{58}
	\|v_1(t)\|_{L^\infty}=\|\partial_xu(t)\|_{L^\infty}<C_1q^{-1}(t),\quad   \text{for\ all}\  t\in [0,T_1], 
	\end{align}
	where \(C_0,C_1\) satisfy \eqref{t3c4}. First observe that 
	\begin{equation*}
	\begin{aligned}
	\|v_0(0)\|_{L^\infty}=\|\phi\|_{L^\infty}<C_0,
	\end{aligned}
	\end{equation*}
	and
	\begin{equation*}
	\begin{aligned}
	\|v_1(0)\|_{L^\infty}=\|\phi^\prime\|_{L^\infty}<C_1q^{-1}(0).
	\end{aligned}
	\end{equation*}
	We then proceed by  contradiction in order  to show \eqref{57} and \eqref{58}. 
	Suppose that \eqref{57} and \eqref{58} hold for all  \(t\in [0, T_2)\), but fails for either \eqref{57} or \eqref{58} at \(t = T_2\) for some \(T_2\in (0, T_1]\). Hence, by continuity, it holds
	\begin{align}\label{59}
	\|v_0(t)\|_{L^\infty}=\|u(t)\|_{L^\infty}<C_0,\quad   \text{for\ all}\  t\in [0,T_2], 
	\end{align}
	and
	\begin{align}\label{60}
	\|v_1(t)\|_{L^\infty}=\|\partial_xu(t)\|_{L^\infty}<C_1q^{-1}(t),\quad   \text{for\ all}\  t\in [0,T_2].
	\end{align}

	 Let \(\eta\in(0,1]\), we split the integral into two parts:
	\begin{equation*}
	\begin{aligned}
	K_0(t,x)&=\underbrace{\int_{|y|<\eta}\frac{\mathrm{sgn}(y)}{|y|^{2+\alpha}}[u(X(t,x),t)-u(X(t,x)-y,t)]\,\diff y}_{I_1}\\
	&\quad+\underbrace{\int_{|y|\geq\eta}\frac{\mathrm{sgn}(y)}{|y|^{2+\alpha}}[u(X(t,x),t)-u(X(t,x)-y,t)]\,\diff y}_{I_2}.
	\end{aligned}
	\end{equation*}
We then estimate
	\begin{equation}\label{61}
	\begin{aligned}
	|I_1|
	&=\bigg|\int_{|y|<\eta}\frac{\mathrm{sgn}(y)}{|y|^{2+\alpha}}[u(X(t,x)-y,t)-u(X(t,x),t)]\,\diff y\bigg|\\
	&\leq |u|_{C^{0,1}(\R)} \int_{|y|<\eta}\frac{1}{|y|^{2+\alpha}}|y|\,\diff y\\
	&\leq \frac{2}{-\alpha}\eta^{-\alpha}\|\partial_xu\|_{L^\infty}\leq \frac{2C_1}{-\alpha}\eta^{-\alpha}q(t)^{-1},
	\end{aligned}
	\end{equation}
and   	
	\begin{equation}\label{61.5}
	\begin{aligned}
	|I_2|
	\leq \frac{4}{1+\alpha}\eta^{-(1+\alpha)}\|v_0\|_{L^\infty}\leq \frac{4C_0}{1+\alpha}\eta^{-(1+\alpha)}.
	\end{aligned}
	\end{equation}

	Choosing \(\eta=q(t)\), for all \(t\in[0,T_2]\) and for all \(x\in\R\), one obtains from \eqref{61} and \eqref{61.5} that
	\begin{equation*}
	\begin{aligned}
	|K_0(t,x)|\leq \bigg(\frac{4C_0}{1+\alpha}+\frac{2C_1}{-\alpha}\bigg)q(t)^{-(1+\alpha)}.
	\end{aligned}
	\end{equation*}
This together with \eqref{54} and \eqref{a7} yields 
	\begin{equation}\label{62}
	\begin{aligned}
	&|v_0(t,x)|\leq \|\phi\|_{L^\infty}+\int_0^{t}|K_0(t,x)|\,\diff t\\
	&\leq \frac{1}{2}C_0+\bigg(\frac{4C_0}{1+\alpha}+\frac{2C_1}{-\alpha}\bigg)\int_0^{t}
	q^{-(1+\alpha)}(\tau)\,\diff \tau\\
	&\leq \frac{1}{2}C_0-\bigg(\frac{4C_0}{1+\alpha}+\frac{2C_1}{-\alpha}\bigg)m^{-1}(0)(1-\delta)^{-(\alpha+2)}(-\alpha)^{-1}[(1-\delta)^{\alpha}-q^{-\alpha}(t)]\\
	&\leq \frac{1}{2}C_0+\frac{1}{\alpha}\bigg(\frac{4C_0}{1+\alpha}+\frac{2C_1}{-\alpha}\bigg)(1-\delta)^{-2}m^{-1}(0)\\
	&< C_0,
	\end{aligned}
	\end{equation}
	for all \(t\in[0,T_2]\) and for all \(x\in\R\), where we have used  \eqref{t3c2}.

	To estimate \(K_1(t,x)\), we also split the integral into two parts:
	\begin{equation*}
	\begin{aligned}
	K_1(t,x)&=\underbrace{\int_{|y|<\eta}\frac{\mathrm{sgn}(y)}{|y|^{2+\alpha}}[\partial_xu(X(t,x),t)-\partial_xu(X(t,x)-y,t)]\,\diff y}_{I_3}\\
	&\quad+\underbrace{\int_{|y|\geq\eta}\frac{\mathrm{sgn}(y)}{|y|^{2+\alpha}}[\partial_xu(X(t,x),t)-\partial_xu(X(t,x)-y,t)]\,\diff y}_{I_4}.
	\end{aligned}
	\end{equation*}
	The term \(I_4\) can be simply estimated  	as
	\begin{equation}\label{63}
	\begin{aligned}
	|I_4|
	\leq \frac{4C_1}{1+\alpha}\eta^{-(1+\alpha)}q^{-1}(t).
	\end{aligned}
	\end{equation}
	For the term \(I_3\), we estimate
	\begin{equation}\label{64}
	\begin{aligned}
	|I_3|
	\leq |\partial_xu|_{C^{0,1}(\R)} \int_{|y|<\eta}\frac{1}{|y|^{2+\alpha}}|y|\,\diff y
	\leq \frac{2}{-\alpha}\eta^{-\alpha}\|\partial_x^2u\|_{L^\infty}.
	\end{aligned}
	\end{equation}
	
	Following the same line as in the proof of \eqref{45}-\eqref{47.5} and using the assumption that \(\delta>0\) is sufficient small (to make the last inequality of \eqref{47} to be true), one has
	\begin{equation}\label{65}
	\begin{aligned}
	\|\partial_x^2u\|_{L^\infty}\leq C_{\mathrm{GN}} \|\partial_xu\|_{L^\infty}^{\frac{1}{3}}\|\partial_x^3u\|_{L^2}^{\frac{2}{3}}\leq 9
	C_{\mathrm{GN}}C_1^{\frac{1}{3}}\|\phi^{\prime\prime\prime}\|_{L^2}^{\frac{2}{3}}q(t)^{-\frac{1}{3}-\frac{7}{3(1-\delta)^2}}.
	\end{aligned}
	\end{equation}
	We then may estimate in view of \eqref{64} and \eqref{65} that
	\begin{equation}\label{66}
	\begin{aligned}
	|I_3|
	\leq \frac{18C_{\mathrm{GN}}C_1^{\frac{1}{3}}\|\phi^{\prime\prime\prime}\|_{L^2}^{\frac{2}{3}}}{-\alpha}
	\eta^{-\alpha}
	q(t)^{-\frac{1}{3}-\frac{7}{3(1-\delta)^2}}.
	\end{aligned}
	\end{equation}

	By choosing \(\eta=q(t)^{-\frac{2}{3}+\frac{7}{3(1-\delta)^2}}\), we conclude from \eqref{63} and \eqref{66} that	
	\begin{equation}\label{67}
	\begin{aligned}
	|K_1(t,x)|
	&\leq \bigg(\frac{4C_1}{1+\alpha}+\frac{18C_{\mathrm{GN}}C_1^{\frac{1}{3}}\|\phi^{\prime\prime\prime}\|_{L^2}^{\frac{2}{3}}}{-\alpha}\bigg)q(t)^{-\big[{\frac{1}{3}+\frac{7}{3(1-\delta)^2}}\big](1+\alpha)+\alpha}\\
	&\leq \bigg(\frac{4C_1}{1+\alpha}+\frac{18C_{\mathrm{GN}}C_1^{\frac{1}{3}}\|\phi^{\prime\prime\prime}\|_{L^2}^{\frac{2}{3}}}{-\alpha}\bigg)q^{-2}(t),
	\end{aligned}
	\end{equation}
	for all \([0, T_2 ]\) and all \(x\in\R\), where we have used
	\begin{equation*}
	\begin{aligned}
	-\bigg[{\frac{1}{3}+\frac{7}{3(1-\delta)^2}}\bigg](1+\alpha)+\alpha\geq -2,
	\end{aligned}
	\end{equation*}
	which follows from  the assumptions that \(\delta>0\) is sufficient small and \(\alpha\in(-1,\frac{5(1-\delta)^2-7}{7-2(1-\delta)^2})\).	
	In view of \eqref{55}, \eqref{67} and \eqref{a7}, one has
	\begin{equation}\label{68}
	\begin{aligned}
	&v_1(t,x)
	\leq\|\phi^\prime\|_{L^\infty}+\bigg(\frac{4C_1}{1+\alpha}+\frac{18C_{\mathrm{GN}}C_1^{\frac{1}{3}}\|\phi^{\prime\prime\prime}\|_{L^2}^{\frac{2}{3}}}{-\alpha}\bigg)\int_0^tq^{-2}(\tau)\,\diff \tau\\
	&\leq \frac{1}{2}C_1q^{-1}(t)-(1-\delta)^{-3}m^{-1}(0)\bigg(\frac{4C_1}{1+\alpha}+\frac{18C_{\mathrm{GN}}C_1^{\frac{1}{3}}\|\phi^{\prime\prime\prime}\|_{L^2}^{\frac{2}{3}}}{-\alpha}\bigg)
	[q^{-1}(t)-(1-\delta)]\\
	&\leq \frac{1}{2}C_1q^{-1}(t)-(1-\delta)^{-3}m^{-1}(0)\bigg(\frac{4C_1}{1+\alpha}+\frac{18C_{\mathrm{GN}}C_1^{\frac{1}{3}}\|\phi^{\prime\prime\prime}\|_{L^2}^{\frac{2}{3}}}{-\alpha}\bigg)q^{-1}(t)\\
	&<C_1q^{-1}(t),
	\end{aligned}
	\end{equation}
	where we have used \eqref{t3c3}.
	On the other hand, one also has
	\begin{equation}\label{69}
	\begin{aligned}
	v_1(t,x)\geq m(t)=m(0)q^{-1}(t)
	\geq-\frac{1}{2}C_1q^{-1}(t).
	\end{aligned}
	\end{equation}

With \eqref{62}, \eqref{68} and \eqref{69} at hand,	the same argument in Section \ref{sec3} can be used to complete the proof.

\end{proof}

\section{Proof of Theorem \ref{th:main-2}}\label{sec6}

\begin{proof}[Proof of Theorem \ref{th:main-2}] It is trivial that the Cauchy problem of \eqref{eq:BH} with \(u(0,x)=\phi(x)\) is well-posed in the class \(C\big([0,T^*):H^2(\R)\big)\) for some \(T^*>0\). We assume that \(T^*\) is the maximal time of existence hereafter.
	Let \(Y(t)\) solve the ODE
	\begin{align*}
	\frac{\mathrm{d}}{\mathrm{d} t}Y(t)=u(t,Y(t)),\quad Y(0)=0.
	\end{align*}
	It is easy to see that \(Y(t)\) is well-defined over \([0,T^*)\). We define 
	\begin{align*}
	v(t,x)=u(t,x+Y(t)).
	\end{align*}
	A small calculation shows that \(v\) satisfies the equation
	\begin{align}\label{eq:BH-new}
	\big(v(t,x)-v(t,0)\big)_t+\big(v(t,x)-v(t,0)\big)v_x(t,x)-\mathcal{H}\big(v(t,x)-v(t,0)\big)=0.
	\end{align}
	
	We introduce the weighted velocity with a smooth, fast decay weight on {\it half line} as follows:
	\begin{align}\label{half line}
	F(t)=-\int_0^\infty \big(v(t,x)-v(t,0)\big)\exp(-x)\, \diff x.
	\end{align}
	 Suppose \(T^*=\infty\).
	It results from \eqref{eq:BH-new} that
	\begin{equation}\label{b3}
	\begin{aligned}
	\frac{\mathrm{d}}{\mathrm{d} t}F(t)
	&=\int_0^\infty\big[\big(v(t,x)-v(t,0)\big)\partial_xv-\mathcal{H}\big(v(t,x)-v(t,0)\big)\big]\exp(-x)\,\diff x.
	\end{aligned}
	\end{equation}
	Since \(v(t,x)-v(t,0)\) vanishes at \(x=0\), by integrating by parts, one has
	\begin{equation}\label{b4}
	\begin{aligned}
	&\int_0^\infty \big(v(t,x)-v(t,0)\big)\partial_xv\exp(-x)\,\diff x\\
	&=\int_0^\infty \big(v(t,x)-v(t,0)\big)\partial_x\big(v(t,x)-v(t,0)\big)\exp(-x)\,\diff x\\
	&=\frac{1}{2}\int_0^\infty \big(v(t,x)-v(t,0)\big)^2\exp(-x)\,\diff x=\colon Q(t).
	\end{aligned}
	\end{equation}
	By the property of Hilbert transform, one obtains
	\begin{equation}\label{b5}
	\begin{aligned}
	&\int_0^\infty\mathcal{H}\big(v(t,x)-v(t,0)\big)\exp(-x)\,\diff x\\
	&\leq \bigg(\int_0^\infty[\mathcal{H}\big(v(t,x)-v(t,0)\big)]^2\,\diff x\bigg)^{\frac{1}{2}}\bigg(\int_0^\infty\exp(-2x)\,\diff x\bigg)^{\frac{1}{2}}\\
	&\leq \frac{\sqrt{2}}{2}\bigg(\int_\R[\mathcal{H}\big(v(t,x)-v(t,0)\big)]^2\,\diff x\bigg)^{\frac{1}{2}}\leq\frac{\sqrt{2}}{2}\|v(t,x)-v(t,0)\|_{L^2}\\
	&=\frac{\sqrt{2}}{2}\|u(t,x+y(t))-u(t,x)\|_{L^2}\leq \sqrt{2}\|u_0\|_{L^2}.
	\end{aligned}
	\end{equation}
	On the other hand, we have
	\begin{equation}\label{b6}
	\begin{aligned}
	F^2(t)
	\leq\int_0^\infty \big(v(t,x)-v(t,0)\big)^2\exp(-x)\,\diff x\int_0^\infty\exp(-x)\,\diff x=2Q(t).
	\end{aligned}
	\end{equation}
	We conclude from \eqref{b3}-\eqref{b6} that
	\begin{align*}
	\frac{\mathrm{d}}{\mathrm{d} t}F(t)\geq \frac{1}{2}F^2(t)-\sqrt{2}\|u_0\|_{L^2}.
	\end{align*}
	In view of \eqref{b1}, it follows that
	\begin{align*}
	F(t)\geq \frac{F(0)}{1-\frac{F(0)}{4}t},
	\end{align*}
	which means that \(F(t)\) will blow up no later than the time \(\frac{4}{F(0)}=\colon T^{**}\) that confirms \eqref{b2}. On the other hand
	\begin{align*}
	F(t)\leq \sup_{x\in\R}\bigg|\frac{v(t,x)-v(t,0)}{x}\bigg|\int_0^\infty x\exp(-x)\, \diff x\leq C\|u_x\|_{L^\infty},
	\end{align*}
	which leads to \eqref{b2.5}.
	This contradiction completes the proof.
\end{proof}

\section{The rescaled Whitham equation}\label{sec7}
We consider here the rescaled Whitham equation \eqref{Whitresc}. We will revisit Theorem \ref{th:W} and its proof by keeping the small parameter $\epsilon,$ our goal being to estimate the blow-up time which should of course be at least of order $\bigO(1/\epsilon)$ since the local solution of the Cauchy problem exists at least on this time scale.

We first need a "rescaled" version of Lemma \ref{le:Hur} :

\begin{lemma} \label{Hur2}There exist constants \(L_0, L_\infty>0\) and \(\eta_0\in(0,1)\) such that
	\begin{equation*}
	\begin{aligned}
	K_\epsilon(x)\leq \epsilon^{-1/4}\frac{L_0}{\sqrt{|x|}}\quad   \mathrm{and}\ |K^\prime(x)|\leq \epsilon^{-1/4}\frac{L_0}{\sqrt{|x|^3}},\quad  \mathrm{for}\   0<|x|<\epsilon^{1/2}\eta_0,
	\end{aligned}
	\end{equation*}
	and
	\begin{equation*}
	\begin{aligned}
	\int_{\epsilon^{1/2}\eta_0}^\infty |K_\epsilon^\prime(x)|\,\diff x\leq \epsilon^{-1/2}L_\infty.
	\end{aligned}
	\end{equation*}	
\end{lemma}

\vspace{0.3cm}
Theorem \ref{th:W} is reformulated as follows:

\begin{theorem}[rescaled Whitham equation]\label{rescWhit} Let \(\delta\in(0,1-\frac{2\sqrt{2}}{3}]\) and \(\epsilon^{-1}\delta<1/2\). Assume that \(\phi\in H^3(\R)\) satisfies
	\begin{equation*}
	\begin{aligned}
	&\delta^2\epsilon^{1/4}(\inf_{\R}\phi^\prime(x))^2> 4L_0C_{\mathrm{Sob}}\|\phi\|_{H^3}+2C_1(3L_0+L_\infty)+36L_0C_{\mathrm{GN}}C_1^{\frac{1}{3}}\|\phi^{\prime\prime\prime}\|_{L^2}^{\frac{2}{3}},\\
	&-(1-\delta)^2\epsilon^{1/4}\inf_{\R}\phi^\prime(x)> 8(3L_0+L_\infty)+16L_0\bigg(\frac{C_1}{C_0}\bigg),\\
	&-(1-\delta)^3\epsilon^{1/4}\inf_{\R}\phi^\prime(x)> 4(3L_0+L_\infty)+36L_0C_{\mathrm{GN}}\bigg(\frac{\|\phi^{\prime\prime\prime}\|_{L^2}}{C_1}\bigg)^{\frac{2}{3}},
	\end{aligned}
	\end{equation*}
	where the constant \(C_0\) and \(C_1\) satisfy 	
	\begin{equation*}
	\begin{aligned}
	\|\phi\|_{L^\infty}\leq\frac{C_0}{2},\quad 
	\|\phi^\prime\|_{L^\infty}\leq\frac{C_1}{2}.
	\end{aligned}
	\end{equation*}	
	Then the solution of the Cauchy problem \eqref{Whitresc} with the initial data \(u(0,x)=\phi(x)\) exhibits wave breaking at some time \(T>0\), namely
	\begin{equation*}
	\begin{aligned}
	|u(x,t)|<\infty,\quad x\in\R, t\in[0,T),
	\end{aligned}
	\end{equation*}	
	but	
	\begin{equation*}
	\begin{aligned}
	\inf_{\R}\partial_xu(x,t)\longrightarrow -\infty,\quad \text{as}\ t\longrightarrow T-.
	\end{aligned}
	\end{equation*}	
	Moreover
	\begin{equation}\label{30.5}
	\begin{aligned}
	-\frac{1}{\inf_{\R}\phi^\prime(x)}\frac{1}{\epsilon(1+\epsilon^{-1}\delta)}<T<-\frac{1}{\inf_{\R}\phi^\prime(x)}\frac{1}{\epsilon(1-\epsilon^{-1}\delta)^2}.
	\end{aligned}
	\end{equation}
	
\end{theorem}

\vspace{0.3cm}
It is worth relating Theorem \ref{rescWhit} to the results in \cite {KLPS} that compare the solution of the rescaled Whitham equation and that of the KdV equation 

\begin{equation}\label{KdV}
\partial_t w+\partial_x w+\epsilon w\partial_x w+ \frac{1}{6}\partial_x^3 w=0.
\end{equation}

Actually the next result is proven in \cite{KLPS} (Theorem 2).

\begin{theorem} \label{compare}
	Let $\phi \in H^{\infty}(\mathbb R)$ and let $u$ and $w$ be the respective solutions of \eqref{Whitresc} and \eqref{KdV} with initial data $\phi.$ Then, for all $j \in \mathbb N$, $j \ge 0$, there exists
	$M_j=M_j(\|\phi\|_{H^{j+8}})>0$  such that
	\begin{equation} \label{maintheo.1}
	\|(u-w)(t)\|_{H^j_x} \le M_j \epsilon^2t,
	\end{equation}
	for all $0 \le t \lesssim \epsilon^{-1}$.
\end{theorem}

\begin{remark}
\rm{The implicit constant in the notation $t \lesssim \epsilon^{-1}$ depends on $\|\phi\|_{H^2}^{-1}$ for $j=0$, $1$ and $\|\phi\|_{H^{j+1}}^{-1}$  for $j\geq 2.$}
\end{remark}

By Theorem \ref{compare}, the solution \(u\) of the rescaled Whitham equation \eqref{Whitresc} cannot 
blow up before a time of order  \(\bigO\big(\epsilon^{-1}\|\phi\|_{H^3}^{-1}\big)\), which is indeed confirmed by \eqref{30.5} since the obtained  wave breaking time is \(\bigO\big(\epsilon^{-1}[-\inf_{\R}\phi^\prime(x)]^{-1}\big)\).

The proof of Theorem \ref{rescWhit} follows that of Theorem \ref{th:W} and Lemma \ref{le:a1}-Lemma \ref{le:a3} by keeping the small parameter $\epsilon$ and we omit it.

\section{Final remarks}

The results in this paper suggest that the fKdV when $-1\leq \alpha<0$ and the Whitham equation share some properties of the nonlinear hyperbolic Burgers equation. One may ask for instance if they possess global weak (entropy) solutions. This has been proven for the Burgers-Hilbert equation in \cite{BN}.

On the other hand those equations have a (weak) dispersive part, with the possibility of existence of global strong small solutions. This is suggested by some numerical simulations in \cite{KLPS} and has been proven for the {\it cubic} fKdV equation with $-1<\alpha<0$ in the work \cite{SW} of the Authors.

 \section{Appendix} 
 
In the proofs of Theorem \ref{th:BH}, \ref{th:W}, \ref{th:fKdV}, and \ref{rescWhit}, we need to handle the following ODE
 \begin{align}\label{a0} 
 \frac{\diff v_1}{\diff t}+\epsilon v_1^2+K_{1,\epsilon}(t,x)=0.
 \end{align}
For the Burgers-Hilbert equation \eqref{eq:BH}, the Whitham equation \eqref{Whit}, and the fKdV equation \eqref{eq:main-1}, one shall take \(\epsilon=1\) in \eqref{a0}, and \(K_{1,\epsilon}(t,x)\\
=\colon K_1(t,x)\) being defined as \eqref{BH-kernal}, \eqref{Whit-kernal}, and \eqref{fKdV-kernal} respectively. For the rescaled Whitham equation \eqref{Whitresc}, one shall keep \(\epsilon\) being a small parameter in \eqref{a0} and define \(K_{1,\epsilon}(t,x)\) by 
 \begin{align*}
 K_{1,\epsilon}(t,x)=\int_\R K_\epsilon(y)\partial_x^2u(X(t,x)-y,t)\,\diff y,
 \end{align*}
where \(K_\epsilon(\cdot)\) is given by \eqref{res-kernal}. 

In the following we always assume \(\delta\in(0,\frac{1}{2})\), \(\epsilon\in (0,1]\), \(\epsilon^{-1}\delta<1/2\), and \(t\in[0,T_1]\), and then let 
\begin{equation*}
\begin{aligned}
	\Sigma_{\delta,\epsilon}(t)=\{x\in\R:v_1(t,x)\leq (1-\epsilon^{-1}\delta)m(t)\},
\end{aligned}
\end{equation*}
and also define
 \[v_1(t,x)=:m(0)r^{-1}(t,x).\]
We denote \(\Sigma_{\delta,1}(t)=\colon \Sigma_\delta(t)\) (\(\epsilon=1\)) for simplicity. 

The following technical lemmas for \(\epsilon=1\) were proved in \cite{NS,HT,Hur}, which can be extended to all \(\epsilon\in (0,1]\) with slight modifications. 
We include the proofs here for the sake of completeness and readers' convenience.
 
 \begin{lemma}\label{le:a1} One has \(\Sigma_{\delta,\epsilon}(t_2)\subset \Sigma_{\delta,\epsilon}(t_1) \) whenever \(0\leq t_1\leq t_2\leq T_1\).
 	
 \end{lemma}
 
 \begin{proof} Suppose that there exists some \(x_1\in\R\) such that
 	 \(x_1\notin\Sigma_{\delta,\epsilon}(t_1)\) but \(x_1\in\Sigma_{\delta,\epsilon}(t_2)\) for some \(0\leq t_1\leq t_2\leq T_1\), that is
 	\begin{equation}\label{a1}
 	\begin{aligned}
 	v_1(t_1,x_1)> (1-\epsilon^{-1}\delta)m(t_1)\quad \text{and}\quad v_1(t_2,x_1)\leq (1-\epsilon^{-1}\delta)m(t_2)<\frac{1}{2}m(t_2).
 	\end{aligned}
 	\end{equation}
 	One can choose \(t_1\) and \(t_2\) close so that
 	\begin{equation*}
 	\begin{aligned}
 	v_1(t,x_1)\leq \frac{1}{2}m(t),\quad \text{for\ all}\ t\in [t_1,t_2].
 	\end{aligned}
 	\end{equation*}
 	Indeed since \(v_1(\cdot,x_1)\) and \(m\) are uniformly continuous on \([0,T_1]\), let 
 	\begin{equation}\label{a2}
 	\begin{aligned}
 	v_1(t_1,x_2)=m(t_1)\leq \frac{1}{2}m(t_1).
 	\end{aligned}
 	\end{equation}
 	We may necessarily choose \(t_2\) close to \(t_1\) so that
 	\begin{equation*}
 	\begin{aligned}
 	v_1(t,x_2)\leq \frac{1}{2}m(t),\quad \text{for\ all}\ t\in [t_1,t_2].
 	\end{aligned}
 	\end{equation*}
 	
 	According to \eqref{9} (\eqref{31} or \eqref{56}), one has
 	\begin{equation*}
 	\begin{aligned}
 	|K_{1,\epsilon}(t,x_j)|\leq \delta^2 m^2(t)\leq 4\delta^2v_1^2(t,x_j)<\frac{1}{2}\delta v_1^2(t,x_j),
 	\end{aligned}
 	\end{equation*}
for all \(t\in[t_1,t_2],\ j=1,2\).
This together with \eqref{a0} yields
 	\begin{equation*}
 	\begin{aligned}
 	\frac{\diff v_1}{\diff t}(\cdot,x_1)=- \epsilon v_1^2(\cdot,x_1)-K_{1,\epsilon}(t,x_1)\geq (-\epsilon-\frac{\delta}{2})v_1^2(\cdot,x_1)
 	\end{aligned}
 	\end{equation*}
 	and
 	\begin{equation*}
 	\begin{aligned}
 	\frac{\diff v_1}{\diff t}(\cdot,x_2)=- \epsilon v_1^2(\cdot,x_2)-K_{1,\epsilon}(t,x_2)\leq (-\epsilon+\frac{\delta}{2})v_1^2(\cdot,x_2),
 	\end{aligned}
 	\end{equation*}
 for \(t\in [t_1,t_2]\). Solving the resulting two inequalities above gives
 	\begin{equation}\label{a2.1}
 	\begin{aligned}
 	v_1(t_2,x_1)\geq\frac{v_1(t_1,x_1)}{1+(\epsilon+\frac{\delta}{2})v_1(t_1,x_1)(t_2-t_1)},
 	\end{aligned}
 	\end{equation}
 	and
 	\begin{equation}\label{a2.2}
 	\begin{aligned}
 	v_1(t_2,x_2)\leq\frac{v_1(t_1,x_2)}{1+(\epsilon-\frac{\delta}{2})v_1(t_1,x_2)(t_2-t_1)}.
 	\end{aligned}
 	\end{equation}
 	
 	 Applying \eqref{a2} to \eqref{a2.2}, one obtains
 	\begin{equation}\label{a2.3}
 	\begin{aligned}
 	m(t_2)\leq\frac{m(t_1)}{1+(\epsilon-\frac{\delta}{2})m(t_1)(t_2-t_1)}.
 	\end{aligned}
 	\end{equation}
 In view of \eqref{a1} and \eqref{a2.1}, one estimates
 	\begin{equation*}
 	\begin{aligned}
 	v_1(t_2,x_1)&>\frac{(1-\epsilon^{-1}\delta)m(t_1)}{1+(\epsilon+\frac{\delta}{2})(1-\epsilon^{-1}\delta)m(t_1)(t_2-t_1)}\\
 	&>\frac{(1-\epsilon^{-1}\delta)m(t_1)}{1+(\epsilon-\frac{\delta}{2})m(t_1)(t_2-t_1)}\\
 	&>(1-\epsilon^{-1}\delta)m(t_2),
 	\end{aligned}
 	\end{equation*}
 where we have used \eqref{a2.3} in the last inequality. We get a contradiction!	

 \end{proof}

 \begin{lemma}\label{le:a2} We have
 	\begin{equation}\label{a3}
 	\begin{aligned}
 \epsilon(1+\epsilon^{-1}\delta)m(0)\leq \frac{\diff r}{\diff t}\leq \epsilon(1-\epsilon^{-1}\delta)m(0),
 	\end{aligned}
 	\end{equation}	
\begin{equation}\label{a4}
\begin{aligned}
q(t)\leq r(t)\leq (1-\epsilon^{-1}\delta)^{-1}q(t),
\end{aligned}
\end{equation}
and
\begin{equation}\label{a5}
\begin{aligned}
0<q(t)\leq 1.
\end{aligned}
\end{equation}

 \end{lemma}
 
 \begin{proof} 
 	Let \(x\in\Sigma_{\delta,\epsilon}(T_1)\), it then follows from Lemma \ref{le:a1} that
 	\begin{equation}\label{a6}
 	\begin{aligned}
 	m(t)\leq v_1(t,x)\leq (1-\epsilon^{-1}\delta)m(t), \quad \text{for\ all}\ t\in [0,T_1].
 	\end{aligned}
 	\end{equation}	
 	The solution of \eqref{a0} can be expressed 	
 	\begin{equation}\label{a6.5}
 	\begin{aligned}
 	v_1(t,x)=\frac{v_1(0,x)}{1+ \epsilon v_1(0,x)\int_0^t\big(1+\epsilon^{-1}v_1^{-2}K_{1,\epsilon}(\tau,x)\big)\,\diff \tau}=m(0)r^{-1}(t,x).
 	\end{aligned}
 	\end{equation}	
It follows from \eqref{a6} and \eqref{9} (\eqref{31} or \eqref{56}) that 
 	\begin{equation*}
 	\begin{aligned}
 	|v_1^{-2}K_{1,\epsilon}(t,x)|\leq (1-\delta)^{-2}\delta^2<\delta,\quad \text{for\ all}\ t\in [0,T_1].
 	\end{aligned}
 	\end{equation*}	
This together with \eqref{a6.5} implies \eqref{a3}. The inequality \eqref{a4} is a consequence of \eqref{a6} and \eqref{a6.5}. It is easy to see that \(r(t,x)\) is decreasing for all \(t\in [0,T_1]\) from \eqref{a3}, and hence \(v_1(t,x)\) too. 
Furthermore, by \eqref{1}, \(q(t)\) is also decreasing for all \(t\in [0,T_1]\)， which implies \eqref{a5} by \eqref{3}.

 \end{proof}
 
 \begin{lemma}\label{le:a3} It holds that
 	\begin{equation}\label{a7}
 	\begin{aligned}
 	\int_0^tq^{-s}(\tau)\,\diff \tau\leq -(1-\epsilon^{-1}\delta)^{-(s+1)}(1-s)^{-1}m^{-1}(0)[(1-\epsilon^{-1}\delta)^{s-1}-q^{1-s}(t)],
 	\end{aligned}
 	\end{equation}	
where\(s>0, s\neq 1\), and	
 	\begin{equation}\label{a8}
 	\begin{aligned}
 		\int_0^tq^{-1}(\tau)\,\diff \tau\leq -(1-\epsilon^{-1}\delta)^{-2}m^{-1}(0)[-\log (1-\epsilon^{-1}\delta)-\log q(t)].
 	\end{aligned}
 	\end{equation}
	
 \end{lemma}

 \begin{proof} 
 	Let \(s>0, s\neq 1\), we use \eqref{a3} and \eqref{a4} to deduce that
 	\begin{equation*}
 	\begin{aligned}
 	&\int_0^tq^{-s}(\tau)\,\diff \tau
 	\leq (1-\epsilon^{-1}\delta)^{-s}\int_0^tr^{-s}(\tau,x)\,\diff \tau\\
 	&\leq (1-\epsilon^{-1}\delta)^{-(s+1)}m^{-1}(0)\int_0^tr^{-s}(\tau,x)\frac{\diff }{\diff t}r(\tau,x)\,\diff \tau\\
 	&=(1-\epsilon^{-1}\delta)^{-(s+1)}(1-s)^{-1}m^{-1}(0)[r^{s-1}(t,x)-r^{1-s}(0,x)],
 	\end{aligned}
 	\end{equation*}	
which combines \eqref{a4} implies \eqref{a7}. One can verify \eqref{a8} similarly.
 	
 \end{proof}

 \vspace{0.5cm}
\noindent {\bf Acknowledgments.} The work of both authors was partially  supported by the ANR project ANuI.


\begin{thebibliography}{000}

\bibitem{BH}
{\sc J. Biello and J. Hunter}, {\it Nonlinear Hamiltonian waves with constant frequency and surface waves on vorticity discontinuity}, Comm. Pure Appl. Math. {\bf 63} (2009), 303-336.

\bibitem{BN}
{\sc A. Bressan and Khai T. Nguyen}, {\it Global exsience of weak solutions for the Burgers-Hilbert equation}, SIAM J. Math. Anal. {\bf46} (2014), 2884-2904.

\bibitem{CCG}
{\sc A. Castro, D. C\'ordoba and F. Gancedo},
{\it Singularity formation in a surface wave model}, 
Nonlinearity, {\bf 23} (2010), 2835--2847.


\bibitem{CE}
{\sc A. Constantin and J. Escher}, {\it Wave breaking for nonlinear nonlocal shallow water equations}, Acta Math. {\bf 181} (1998), 229-245. 
	
\bibitem{EW}
{\sc M.Ehrnstr\"{o}m and Y. Wang}, \textsc{\it Enhanced existence time of solutions to the fractional Korteweg-de Vries equation}, SIAM J. Math. Anal. {\bf 51} (2019), pp.~3298-3323.	
	
\bibitem{Hur1}
{\sc V. Hur}, {\it On the formation of singularities for surface water waves},  Commun. Pure Appl. Anal.  {\bf 11} (2012),  1465-1474.	
	
\bibitem{HT}
{\sc V. Hur and L Tao}, {\it Wave breaking for the Whitham equation with fractional dispersion}, Nonlinearity {\bf 27} (2014), 2937-2949.

	
\bibitem{Hur}
{\sc V. Hur}, {\it Wave breaking in the {W}hitham equation}, Adv. Math. {\bf 317} (2017), 410-437.
	
\bibitem{HI} 
{\sc J. Hunter and M. Ifrim},
{\it Enhanced lifespan of smooth solutions of a Burgers-Hilbert equation}, 
SIAM J. Math. Anal. {\bf 44} (2012), 2039-2052.


\bibitem{HITW} 
{\sc J. Hunter, M. Ifrim, D. Tataru and T. Wong}, {\it Long time solutions for a Burgers -Hilbert equation via a modified energy method}, Proc; Amer. Math. Soc.
{\bf 143} (2015), 3407-3412.

 \bibitem{KLPS}
{\sc C. Klein, F. Linares, D. Pilod and J.-C. Saut}, {\it On Whitham and related equations},  Studies in Appl. Math.  {\bf 140} (2018), 133-177

\bibitem{MR3317254}
{\sc C. Klein, and J.-C. Saut}, {\it A numerical approach to blow-up issues for dispersive perturbations of {B}urgers' equation}, Phys. D {\bf 295/296} (2015), pp.~46--65.
	
\bibitem{LPS2}
{\sc F. Linares, D. Pilod, and J.-C. Saut},
{\it Dispersive perturbations of Burgers and hyperbolic equations I: local
theory}, SIAM J. Math. Anal., {\bf 46} (2014), 1505-1537.

\bibitem{LW}
{\sc H. Li, and Y. Wang}, \textsc{\it Formation of singularities of spherically symmetric solutions to the 3D compressible Euler equations and Euler-Poisson equations}, NoDEA Nonlinear Differential Equations Appl. {\bf 25} (2018), 15pp.


\bibitem{MPV} 
{\sc L. Molinet, D. Pilod and S. Vento}, {\it On well-posedness for some dispersive perturbations of the Burgers equation}, Ann. Inst. H. Poincar\'e Anal. Non Lin. {\bf 35} (2018), 1719-1756. 
	
\bibitem{NS}
{\sc P. I. Naumkin and I. A. Shishmar\"{e}v}, {\it Nonlinear nonlocal equations in the theory of waves}, Translations of Mathematical Monographs {\bf 133} (1994), American Mathematical Society, Providence, RI, Translated from the Russian manuscript by Boris Gommerstadt.
	
\bibitem{SW}
{\sc J.-C. Saut and Yuexun Wang}, {\it Long time behavior of the fractional Korteweg-de Vries equation with cubic nonlinearity},  arXiv:2003.05910, (2020).
	
\bibitem{W}\textsc{G.B. Whitham}, {\it Linear and nonlinear waves}, Wiley, New York 1974.
	
\bibitem{Whi}
{\sc G. B. Whitham}, {\it Variational methods and applications to water waves}, Proc.R. Soc. Lond. Ser. A., {\bf 299} (1967), 6-25.

	
\end{thebibliography}
\end{document}